\documentclass[11pt]{amsart}
\RequirePackage[dvipsnames,usenames]{color}
\usepackage{amssymb,amsthm,amsmath,amscd,quiver}

\usepackage{mathtools}
\usepackage{soul}
\usepackage[backref=page, colorlinks=true, linkcolor=magenta, citecolor=cyan]{hyperref}
\usepackage{marvosym}

\input{kmacros3.sty}

\usepackage{tikz}
\usepackage{tikz-cd}
\usetikzlibrary{cd}
\usepackage{graphicx}
\usepackage[all,cmtip]{xy}
\usepackage{mathpazo}

\usepackage{mabliautoref}
\numberwithin{equation}{theorem}

\usepackage{bm}
\usepackage{pifont}
\usepackage{upgreek}

\usepackage{eucal}
\usepackage{ulem}
\usepackage{stmaryrd}

\numberwithin{equation}{theorem}

\usepackage{fullpage}

\usepackage{setspace}

\usepackage{enumerate}
\usepackage{calc}

\usepackage{verbatim}
\usepackage{alltt}

\makeatletter
\def\@tocline#1#2#3#4#5#6#7{\relax
  \ifnum #1>\c@tocdepth 
  \else
    \par \addpenalty\@secpenalty\addvspace{#2}%
    \begingroup \hyphenpenalty\@M
    \@ifempty{#4}{%
      \@tempdima\csname r@tocindent\number#1\endcsname\relax
    }{%
      \@tempdima#4\relax
    }%
    \parindent\z@ \leftskip#3\relax \advance\leftskip\@tempdima\relax
    \rightskip\@pnumwidth plus4em \parfillskip-\@pnumwidth
    #5\leavevmode\hskip-\@tempdima
      \ifcase #1
       \or\or \hskip 1em \or \hskip 2em \else \hskip 3em \fi%
      #6\nobreak\relax
    \hfill\hbox to\@pnumwidth{\@tocpagenum{#7}}\par
    \nobreak
    \endgroup
  \fi}

\makeatother

\begin{document}

\title{A characteristic \MakeLowercase{\emph{p}} analog of formal lifting properties}
\author{Rankeya Datta}
\email{rankeya.datta@missouri.edu}
\thanks{}
\author{Noah Olander}
\email{nolander@berkeley.edu}

\begin{abstract}
   A field extension $L/K$ of characteristic $p > 0$ is formally \'etale if and only if the relative Frobenius of $L/K$ is an isomorphism. Inspired by this classical result, we explore whether the formally \'etale property for a map $R \to S$ of $\mathbf{F}_p$-algebras is characterized by isomorphism of the relative Frobenius $F_{S/R}$. While $F_{S/R}$ being an isomorphism implies $R \to S$ is formally \'etale, the converse fails in the non-Noetherian setting. Thus, following Morrow, we introduce an enhancement of the formally \'etale property that we call b-nil (bounded nil) formally \'etale, and we show that $F_{S/R}$ is an isomorphism precisely when $R \to S$ is b-nil formally \'etale. We prove this result by first establishing several structural properties of b-nil formally smooth maps, which are defined analogously to the formally smooth case.  Our structural results reveal that the b-nil formally smooth (resp. \'etale) property is quite different from the formally smooth (resp. \'etale) property. For instance, we show that any b-nil formally smooth algebra over an $F$-pure ring is reduced, whereas non-reduced formally \'etale algebras exist over $\mathbf{F}_p$ by a construction of Bhatt. We also show that the b-nil formally \'etale property neither implies nor is implied by having a trivial cotangent complex.   We explore when formally smooth (resp. \'etale) implies b-nil formally smooth (resp. \'etale) in prime characteristic. A satisfactory picture emerges for ideal adic completions.
\end{abstract}
\maketitle

\setcounter{tocdepth}{2}
\tableofcontents

\section{Introduction}
A ring homomorphism $R \to S$ is \emph{formally \'etale} if for all pairs $(A,I)$, where $A$ is an $R$-algebra and $I \subset A$ is an ideal such that $I^2 = 0$, and for all commutative diagrams of rings
\[\begin{tikzcd}[cramped]
	{A/I} & S \\
	A & R
	\arrow["f"', from=1-2, to=1-1]
	\arrow[two heads, from=2-1, to=1-1]
	\arrow[from=2-2, to=1-2]
	\arrow[from=2-2, to=2-1]
\end{tikzcd}\]
where $A \twoheadrightarrow A/I$ is the canonical projection, there exists a unique $R$-algebra map $\widetilde{f} \colon S \to A$ that lifts $f$ along $A \twoheadrightarrow A/I$, i.e., $\widetilde{f}$ makes the triangles in the following diagram commute:
\[\begin{tikzcd}[cramped]
	{A/I} & S \\
	A & R
	\arrow["f"',from=1-2, to=1-1]
	\arrow["{\exists!\widetilde{f}}"{description}, color={rgb,255:red,214;green,92;blue,92}, dashed, from=1-2, to=2-1]
	\arrow[two heads, from=2-1, to=1-1]
	\arrow[from=2-2, to=1-2]
	\arrow[from=2-2, to=2-1].
\end{tikzcd}\]
Thus, a formally \'etale ring map is \emph{formally unramified} (which is the condition that an $R$-algebra lift of $f$ is unique) and \emph{formally smooth} (which is the condition that an $R$-algebra lift of $f$ exists).

The formally \'etale property is already interesting for a field extension $L/K$. When $\Char(K) = 0$, $L/K$ is formally \'etale if and only if $L/K$ is separable algebraic. The situation is more delicate when $\Char(K) = p > 0$. Indeed, while separable algebraic field extensions are formally \'etale in any characteristic, if $L/K$ is an extension of perfect fields of characteristic $p$, then $L/K$ is always formally \'etale but it is not necessarily algebraic or separably generated. Nevertheless, one has the following classical characterization.

\begin{theorem}\cite[Theorem~26.7]{MatsumuraCommutativeRingTheory}
    \label{thm:formally-etale-fields}
    Let $L/K$ be an extension of fields of characteristic $p > 0$. The following are equivalent:
    \begin{enumerate}
        \item $L/K$ is formally \'etale.
        \item Any absolute $p$-basis of $K$ is also an absolute $p$-basis of $L$.
        \item The canonical map $K \otimes_{K^p} L^p \to L$ is an isomorphism.
    \end{enumerate}
\end{theorem}

The map $K \otimes_{K^p} L^p \to L$ can be identified with the relative Frobenius of $L/K$. More generally, if $\varphi \colon R \to S$ is a homomorphism of $\mathbf{F}_p$-algebras, then the \emph{relative Frobenius of $\varphi$}, denoted $F_\varphi$, is the unique ring map $F_*R \otimes_R S \to F_*S$ such that the following diagram commutes:
 \[
   \begin{tikzcd}[cramped]
	R && S \\
	{F_*R} && {F_*R\otimes_RS} \\
	&&& {F_*S}
	\arrow["\varphi", from=1-1, to=1-3]
	\arrow["{F_R}"', from=1-1, to=2-1]
	\arrow["{F_R\otimes_R\id_S}", from=1-3, to=2-3]
	\arrow["{F_S}", curve={height=-24pt}, from=1-3, to=3-4]
	\arrow["{\id_{F_*R}\otimes_R\varphi}"', from=2-1, to=2-3]
	\arrow["{F_*\varphi}"', curve={height=24pt}, from=2-1, to=3-4]
	\arrow["{\exists! F_\varphi}"{description}, color={rgb,255:red,214;green,92;blue,92}, from=2-3, to=3-4].
\end{tikzcd}
\]
 Here $F_R$ (resp. $F_S$) is the absolute Frobenius ($p$-th power) endomorphism of $R$ (resp. $S$), and $F_*R$ is the ring $R$ viewed as an $R$-module by restriction of scalars along $F_R$ (and similarly for $F_*S$).

 Inspired by \autoref{thm:formally-etale-fields} one can ask the following natural question (c.f. \cite[Question~3.9]{Bhattp-adic}):

 \begin{question}
     If $\varphi \colon R \to S$ is a map of $\mathbf{F}_p$-algebras, does $F_\varphi$ being an isomorphism characterize $\varphi$ being formally \'etale?
 \end{question}

If $F_\varphi$ is an isomorphism, then $\varphi$ is formally \'etale; see \autoref{thm:relative-Frobenius-iso-b-nil formally-etale} and \autoref{cor:relative-Frobenius-iso-implies-formally-etale} for an elementary proof along the lines of the one given in \cite{MatsumuraCommutativeRingTheory} for field extensions\footnote{$F_\varphi$ being an isomorphism at the level of rings does not imply the cotangent complex $L_{\varphi}$ is trivial (\autoref{eg:formally-etale-not-b-nil formally-etale}~\autoref{eg:formally-etale-not-b-nil formally-etale.c}), although $H^0(L_\varphi) = H^{-1}(L_\varphi) = 0$ by formal \'etaleness \cite[\href{https://stacks.math.columbia.edu/tag/0D0M}{Tag 0D0M}, \href{https://stacks.math.columbia.edu/tag/08RB}{Tag 08RB}]{stacks-project}.}.  However, we give an example  of a formally \'etale map of semi-perfect rings (i.e. rings with surjective absolute Frobenius) for which the relative Frobenius is not an isomorphism; see \autoref{eg:formally-etale-not-b-nil formally-etale}~\autoref{eg:formally-etale-not-b-nil formally-etale.b}. This raises the natural question of whether isomorphism of the relative Frobenius corresponds to a lifting property stronger than the one used to define formally \'etale maps. We show that such a strengthening indeed exists, and we call the corresponding notion \emph{b-nil formally \'etale}. 

\begin{theorem}(\autoref{thm:relative-Frobenius-iso-b-nil formally-etale})
    \label{thm:main-thm-intro}
    Let $\varphi \colon R \to S$ be a map of $\mathbf{F}
    _p$-algebras. The following are equivalent:
    \begin{enumerate}
        \item $\varphi$ is b-nil formally \'etale.
        \item The relative Frobenius $F_{\varphi}$ is b-nil formally \'etale.
        \item The relative Frobenius $F_{\varphi}$ is b-nil formally smooth.
        \item The relative Frobenius $F_{\varphi}$ is an isomorphism.
    \end{enumerate} 
\end{theorem}

 We learned of the b-nil formally \'etale property from notes by Morrow; see \cite[Project~C]{Morrow2019THH}. The term `b-nil' is an abbreviation of \emph{bounded nil}; see \autoref{def:nil-ideal}. In prime characteristic $p > 0$, $R \to S$ is \emph{b-nil formally \'etale} precisely when for all pairs $(A,I)$, where $A$ is an $R$-algebra and $I \subset A$ is an ideal such that $I^{[p]} = 0$, and for every solid commutative square
\[
\begin{tikzcd}[cramped]
	{A/I} & S \\
	A & R
	\arrow[from=1-2, to=1-1]
	\arrow[color={rgb,255:red,214;green,92;blue,92}, dashed, from=1-2, to=2-1]
	\arrow[two heads, from=2-1, to=1-1]
	\arrow[from=2-2, to=1-2]
	\arrow[from=2-2, to=2-1],
\end{tikzcd}
\]
there is a unique dashed arrow $S \to A$ that makes the triangles commute. Here $I^{[p]}$ is the ideal generated by the $p$-th powers of elements of $I$, i.e., $I^{[p]}$ is the expansion of $I$ along the absolute Frobenius of $A$. The notions of \emph{b-nil formally unramified} and \emph{b-nil formally smooth} ring maps are defined analogously. Note that the example of a formally \'etale map that does not have isomorphic relative Frobenius (\autoref{eg:formally-etale-not-b-nil formally-etale}~\autoref{eg:formally-etale-not-b-nil formally-etale.b}) gives an example of a formally \'etale map that is \emph{not} b-nil formally \'etale.  

To prove \autoref{thm:main-thm-intro}, we initiate a study of b-nil formal smoothness drawing inspiration from the work of Gabber \cite{Gabber.tStruc}. We obtain the following characterization.

\begin{theorem}(\autoref{prop:b-nil-formally-smooth-injective-relative-Frobenius-projective}, \autoref{thm:characterizing-b-nil-formal-smoothness})
    \label{thm:intro-b-nil-formal-smoothness}
    Let $\varphi \colon R \to S$ be a map of $\mathbf{F}_p$-algebras. Let $S^{(p)} \coloneqq F_*R \otimes_R S$. Then the following are equivalent:
    \begin{enumerate}
        \item $\varphi$ is b-nil formally smooth.
        \item There exists $\{f_i\}_{i \in I} \subset S$ such that the canonical $S^{(p)}$-algebra map 
        \[\frac{S^{(p)}[T_i \colon i \in I]}{(T_i^p - (1\otimes f_i) \colon i \in I)} \to S\] 
        that acts on $S^{(p)}$ via $F_\varphi$ and sends the class of $T_i$ to $f_i$ admits an $S^{(p)}$-algebra right-inverse. 
    \end{enumerate}
    In particular, if $R \to S$ is b-nil formally smooth, then 
    $S$ is a projective $S^{(p)}-$module via $F_{\varphi}$ and so $F_{\varphi}$ is faithfully flat and has an $S^{(p)}$-linear left inverse.
\end{theorem}


\noindent This result is reminiscent of the well-known fact that a formally smooth map of Noetherian rings in characteristic $p$ has faithfully flat relative Frobenius (see \autoref{thm:Radu-Andre}). \autoref{thm:intro-b-nil-formal-smoothness} also indicates that b-nil formal smoothness of $\varphi \colon R \to S$ is `close' to injectivity of $F_\varphi$ and $S$ having an $p$-basis over $R$. In fact, by a straightforward application of \autoref{thm:intro-b-nil-formal-smoothness}, we deduce that these last two conditions imply b-nil formal smoothness of $\varphi$ (\autoref{cor:p-basis-implies-bnil-formally-smooth}). This improves \cite[$0_{\text{IV}}$,~Th\'eor\`eme~21.2.7]{EGAIV_I}, which shows that $\varphi$ is formally smooth assuming injectivity of $F_\varphi$ and the existence of a $p$-basis of $S$ over $R$. Furthermore, \autoref{thm:intro-b-nil-formal-smoothness} allows us to identify a large class of maps for which the formally smooth and b-nil formally smooth properties coincide.

\begin{theorem}(\autoref{thm:b-nil-formally-smooth-fp-relFrob})
\label{thm:intro-b-nil-formally-smooth-fp-relFrob}
    Let $\varphi \colon R \to S$ be a map of $\mathbf{F}_p$-algebras such that $F_\varphi$ is finitely presented (for e.g., if $F_R$ is finite and $F_S$ is finitely presented). Then $\varphi$ is b-nil formally smooth if and only if $\varphi$ is formally smooth.
\end{theorem}

Another consequence of \autoref{thm:intro-b-nil-formal-smoothness} is that if $R$ has universally injective absolute Frobenius (i.e. $R$ is \emph{$F$-pure}) and $R \to S$ is b-nil formally smooth, then $S$ is reduced (\autoref{cor-geometricallyreduced}~\autoref{cor-geometricallyreduced.a}). Since the b-nil formally smooth property is preserved under base change, this shows that any b-nil formally smooth algebra over a field of prime characteristic is geometrically reduced (\autoref{cor-geometricallyreduced}~\autoref{cor-geometricallyreduced.b}). This is in stark contrast with Bhatt's example of an $\mathbf{F}_p$-algebra with trivial cotangent complex, hence formally \'etale over $\mathbf{F}_p$, that is \emph{not} reduced \cite{Bhatt_imperfect_trivial_cc}, exhibiting a further difference between the formally \'etale and b-nil formally \'etale properties. Our reducedness result combined with Bhatt's example shows that the triviality of the cotangent complex of a map of $\mathbf{F}_p$-algebras does not imply its b-nil formal \'etaleness. Furthermore, b-nil formally \'etale maps, like formally \'etale ones, need not have trivial cotangent complex; see \autoref{eg:formally-etale-not-b-nil formally-etale}~\autoref{eg:formally-etale-not-b-nil formally-etale.c}. In \autoref{sec:b-nil-formally-smooth-flat-Frobenius} we also examine the relationship between b-nil formally smooth maps in prime characteristic and some enhancements of rings with flat absolute Frobenius that appear in the theories of $F$-singularities and tight closure \cite{HochsterHunekeFRegularityTestElementsBaseChange, KatzmanParameterTestIdealOfCMRings, KatzmanLyubeznikZhangOnDiscretenessAndRationality, SharpAnExcellentFPureRing, SharpBigTestElements, DESTate}. \autoref{thm:intro-b-nil-formally-smooth-fp-relFrob} allows us to improve results from \cite{DESTate} on when these enhancements of flat Frobenius ascend along `nice' maps.  

The divergence of the formally \'etale and b-nil formally \'etale properties naturally leads us to examine situations where they agree in prime characteristic. Several positive results have long been known, mainly following from the fact that the relative Frobenius of a weakly \'etale map of $\mathbf{F}_p$-algebras is an isomorphism; see \autoref{sec:Section-3} for more details. We record two additional cases that may be of interest. The first is the analog of \autoref{thm:intro-b-nil-formally-smooth-fp-relFrob} for the (b-nil) formally \'etale property.

 \begin{theorem}(\autoref{thm:formally-etale-F-finite})
     \label{thm:F-finite-case}
     Let $\varphi \colon R \to S$ be a map of $\mathbf{F}_p$-algebras such that $F_\varphi$ is finitely presented. Then $\varphi$ is b-nil formally \'etale if and only if $\varphi$ is formally \'etale.
 \end{theorem}
 \autoref{rem:following-bnil-formally-etale-Ffinite} shows that one cannot weaken the hypothesis of finite presentation of $F_\varphi$ to finite type in both \autoref{thm:intro-b-nil-formally-smooth-fp-relFrob} and \autoref{thm:F-finite-case}. The proof of \autoref{thm:F-finite-case} is similar to the classical proof that for a  finitely presented ring map, isomorphism of the relative Frobenius characterizes the \'etale property \cite[Expos\'e~XV, $n^0~2$, Proposition~2~c)]{SGA5}. Indeed, the hypotheses of \autoref{thm:F-finite-case} imply that $F_\varphi$ is \'etale. The usual reasoning that an \'etale radicial morphism of schemes is an open immersion then implies $\Spec(F_\varphi)$ is an isomorphism of affine schemes, i.e., $F_\varphi$ is an isomorphism of rings.

 Our second result, in the Noetherian setting, is the following unconditional equivalence.

\begin{theorem}(\autoref{thm:formally-etale-Noetherian-relative-Frobenius-iso})
\label{thm:Noetherian-main-thm-intro}
    Let $\varphi \colon R \to S$ be a map of Noetherian $\mathbf{F}_p$-algebras. Then $\varphi$ is b-nil formally \'etale if and only if $\varphi$ is formally \'etale.
\end{theorem}

\noindent Given \autoref{thm:main-thm-intro} and the fact that b-nil formally \'etale implies formally \'etale, the main assertion of \autoref{thm:Noetherian-main-thm-intro} is that a formally \'etale map of Noetherian rings has isomorphic relative Frobenius. We were unable to locate an elementary proof of this result in the literature, although it is presumably well-known to experts\footnote{After the first version of our paper was posted to arXiv, we learned that Javier Carvajal-Rojas and Axel St\"abler had independently proved this result \cite{carvajalrojas2025kunztypetheoremformalunramification} using a result of Andr\'e and Ma--Polstra on surjectivity of formally unramified maps of Noetherian $\mathbf{F}_p$-algebras that are invertible up to a power of Frobenius \cite[Theorem~11.2]{MaPolstraFSing}; see \autoref{rem:concluding-remarks-main-thm-Noetherian}.}.  In \cite[Proposition 3.5.6]{LurieEllipticII}, Lurie uses the machinery of \cite{LurieSpectralAlgebraicGeometry} to show that if the $\mathbb{E}_\infty$-cotangent complex of a map of Noetherian $\mathbf{F}_p$-algebras is trivial, then the relative Frobenius is an isomorphism. 
Additionally, formally \'etale implies b-nil formally \'etale for Noetherian $F$-finite rings by \autoref{thm:F-finite-case}, or by recent work of Fink \cite[Theorem~5.3]{FinkRelativelyPerfect}\footnote{Fink studies the relative Frobenius of maps of animated $\mathbf{F}_p$-algebras, and his proof that a formally \'etale map of classical $F$-finite Noetherian $\mathbf{F}_p$-algebras has isomorphic relative Frobenius appears to use this derived perspective.}. 

The first ingredient in our proof of \autoref{thm:Noetherian-main-thm-intro} is that a formally smooth map of Noetherian rings is flat with geometrically regular fibers, i.e., a \emph{regular} map. This result is due to Grothendieck \cite{EGAIV_I}. The second ingredient, due to Radu \cite{RaduUneClasseDAnneaux} and Andr\'e \cite{AndreHomomorphismsRegulariers}, is that a map of Noetherian rings of characteristic $p$ is regular if and only if its relative Frobenius is faithfully flat. This already implies the injectivity of $F_\varphi$ assuming that $\varphi$ is formally \'etale. We prove surjectivity first in the case where $R$ or $S$ is reduced (\autoref{prop:formally-etale-Noetherian-reduced}) and then get the general result by a deformation argument. 

It would be interesting to give a proof of \autoref{thm:Noetherian-main-thm-intro} that uses just the Noetherian property to show that formally \'etale maps are b-nil formally \'etale directly via the definitions of these notions. We have been unsuccessful in finding an argument that works for arbitrary Noetherian rings. However, when the Noetherian rings are additionally $F$-finite, see \autoref{prop:equivalences-b-nil-F-finite-case} which offers another perspective on why the various formal properties coincide.  Finally, in \autoref{sec:formal-smooth-completion} we study the relationship between (b-nil) formal smoothness/\'etaleness and regularity of the completion map of a Noetherian ring with respect to an arbitrary ideal. If $(R,\m)$ is an excellent local ring, then the completion map $R \to \widehat{R}$ is regular, or equivalently, ind-smooth by N\'eron-Popescu desingularization. However, $R \to \widehat{R}$ fails to be formally smooth even in very simple examples; see \autoref{rem:very-rare-char0} for a characteristic $0$ example and \autoref{rem:infinite-field-char-p-not-formally-smooth} for a characteristic $p$ one. This also clarifies a common misconception that ind-smooth implies formally smooth (ind-formally \'etale implies formally \'etale which is perhaps why this misconception exists). However, in prime characteristic we prove the somewhat surprising fact that ideal adic completion being a regular map is equivalent to its (b-nil) formal smoothness/\'etaleness with an appropriate $F$-finiteness assumption that is common in geometric situations. Our result, which appears below, is just one among many others that illustrate the special nature of $F$-finite Noetherian rings. Some salient examples are Kunz's result that any $F$-finite Noetherian ring is excellent of finite Krull dimension \cite{KunzOnNoetherianRingsOfCharP}, and Gabber's result that any $F$-finite Noetherian ring admits a dualizing complex because it is a homomorphic image of an $F$-finite regular ring \cite{Gabber.tStruc}.

\begin{proposition}(\autoref{prop:completionforF-finiteGring})
\label{prop:intro-completionforF-finiteGring}
    Let $R$ be a Noetherian ring of prime characteristic $p > 0$. Let $I$ be an ideal of $R$ such that $R/I$ is $F$-finite. Let $\widehat{R}$ denote the $I$-adic completion of $R$. Then we have the following:
    \begin{enumerate}
    \item $F_{\widehat{R}/R}$ is surjective, and hence, $R \to \widehat{R}$ is b-nil formally unramified.
    \item  The following are equivalent:
        \begin{enumerate}[(i)]
        \item $R \to \widehat{R}$ is a regular map.
        \item $R \to \widehat{R}$ is b-nil formally \'etale.
        \item $R \to \widehat{R}$ is b-nil formally smooth.
        \item $R \to \widehat{R}$ is formally smooth.
        \item $R \to \widehat{R}$ is formally \'etale.
        \end{enumerate}
    \end{enumerate}
\end{proposition}

We also establish a non-Noetherian analog of \autoref{prop:intro-completionforF-finiteGring} for completions of arbitrary rings with respect to finitely generated ideals and show that for prime characteristic Noetherian rings $R$, even when $R/I$ is \emph{not} $F$-finite, the properties b-nil formally smooth, b-nil formally \'etale, formally smooth and formally \'etale coincide for the $I$-adic completion map $R \to \widehat{R}$; see \autoref{prop:equivalence-completions-arbitrary}. However, the point is that without the $F$-finiteness of $R/I$, $R \to \widehat{R}$ being regular need not imply these nice formal properties.


\section{Preliminaries}
When we use the terms formally smooth/\'etale/unramified we always mean with respect to the discrete topology. This is what Matsumura calls $0$-smooth/\'etale/unramified in \cite{MatsumuraCommutativeRingTheory}.

Fix a prime number $p$. For an $\mathbf{F}_p$-algebra $R$ we have the \emph{absolute Frobenius endomorphism} of $R$
\begin{align*}
    F_R \colon R &\to R\\
    r &\mapsto r^p
\end{align*}
and its iterates $F^e_R$, for any integer $e > 0$. One often considers the target copy of $R$ as an $R$-module via restriction of scalars along $F^e_R$, in which case it is common to denote the resulting $R$-algebra by $F^e_*R$ (some authors choose $\prescript{(e)}{}{R}$ or even $R^{1/p^e}$). An $\mathbf{F}_p$-algebra $R$ is \emph{$F$-pure} if $F_R$ is a universally injective ring map, it is \emph{$F$-finite} if $F_R$ is finite, it is \emph{semi-perfect} if $F_R$ is surjective, and it is \emph{perfect} if $F_R$ is an isomorphism.

If $\varphi \colon R \to S$ is a map of $\mathbf{F}_p$-algebras, then one also has a relative version of the $e$-th absolute Frobenius called the \emph{$e$-th relative Frobenius}, denoted $F^e_\varphi$, which is the unique ring map $F^e_*R \otimes_R S \to F^e_*S$ that makes the following diagram commute:
 \[
   \begin{tikzcd}[cramped]
	R && S \\
	{F^e_*R} && {F^e_*R\otimes_RS} \\
	&&& {F^e_*S}
	\arrow["\varphi", from=1-1, to=1-3]
	\arrow["{F^e_R}"', from=1-1, to=2-1]
	\arrow["{F^e_R\otimes_R\id_S}", from=1-3, to=2-3]
	\arrow["{F^e_S}", curve={height=-24pt}, from=1-3, to=3-4]
	\arrow["{\id_{F^e_*R}\otimes_R\varphi}"', from=2-1, to=2-3]
	\arrow["{F^e_*\varphi}"', curve={height=24pt}, from=2-1, to=3-4]
	\arrow["{\exists! F_\varphi}"{description}, color={rgb,255:red,214;green,92;blue,92}, from=2-3, to=3-4]
\end{tikzcd}
\]
The ring $F^e_*R\otimes_R S$ is often also denoted as $S^{(p^e)}$ (or $S^{(e)}$ by some authors) and $F^e_\varphi$ is considered as a ring homomorphism $S^{(p^e)} \to S$, i.e., one typically suppresses the $F^e_*$ notation. Note that the $R$-algebra structure on $S^{(p^e)}$ is the $F^e_*R$-algebra structure. In this paper, we will switch between dropping and using the $F_*$'s when discussing the absolute/relative Frobenii. We hope this will not cause any confusion for the reader.

One advantage of working with the relative Frobenius is that it commutes with base change whereas the absolute Frobenius does not. Namely, if $\varphi \colon R \to S$ and $\phi \colon R \to T$ are $\mathbf{F}_p$-algebras, then $\id_T \otimes_R F_\varphi$ can be identified with $F_\phi$; see \cite[Expose~XV, $n^0 2$, Proposition~1 ~b)]{SGA5}. Another advantage is that the relative Frobenius of a map of Noetherian $\mathbf{F}_p$-algebras controls the singularities of the fibers of the map, as summarized in the result below.

\begin{theorem}
    \label{thm:Radu-Andre}
    Let $\varphi \colon R \to S$ be a map of Noetherian $\mathbf{F}_p$-algebras. Then:
    \begin{enumerate}
        \item $\varphi$ is flat with geometrically regular fibers (i.e., $\varphi$ is a regular map) if and only if $F_\varphi$ is faithfully flat.\label{thm:Radu-Andre.a}
        \item $\varphi$ is flat with geometrically reduced fibers if and only if $\varphi$ is flat and $F_\varphi$ is universally injective as a map of $R$-modules (i.e. $F_*R$-modules).\label{thm:Radu-Andre.b}
        \item If $\varphi$ is formally smooth, then $\varphi$ is a regular map. Hence $F_\varphi$ is a faithfully flat map of Noetherian rings.\label{thm:Radu-Andre.c}
    \end{enumerate}
\end{theorem}

\begin{proof}
    \autoref{thm:Radu-Andre.a} is due to Radu \cite{RaduUneClasseDAnneaux} and Andr\'e \cite{AndreHomomorphismsRegulariers}. See also \cite[Theorem~10.1]{MaPolstraFSing} for an accessible proof of \autoref{thm:Radu-Andre.a}. \autoref{thm:Radu-Andre.b} is due to Dumitrescu \cite{DumitrescuReduceness}. The fact that a formally smooth map of Noetherian rings is regular is a deep result of Grothendieck \cite[$0_{IV}$, Corollaire~19.6.5, Th\`eor\'eme~19.7.1]{EGAIV_I}. An alternate reference is \cite[\href{https://stacks.math.columbia.edu/tag/07NQ}{Tag 07NQ}]{stacks-project} (note formal smoothness with respect to linear topologies is implied by formal smoothness \cite[\href{https://stacks.math.columbia.edu/tag/07EC}{Tag 07EC}]{stacks-project}). The faithful flatness in the last assertion of \autoref{thm:Radu-Andre.c} follows from \autoref{thm:Radu-Andre.a}. Note $S^{(p)}$ is Noetherian as a result of the faithful flatness of $F_\varphi$ because the target $S$ is Noetherian, and the property of being Noetherian satisfies faithfully flat (in fact, even cyclically pure) descent.
\end{proof}

\begin{remark}
    Another proof of \autoref{thm:Radu-Andre}~\autoref{thm:Radu-Andre.c} is given in \cite[Theorem~2.11]{BlickleFinkFfinite} using deep rigidity properties of the cotangent complex of a map of Noetherian rings due to Briggs-Iyengar \cite{BriggsIyengarCotangentRigidity}, Avramov \cite{AvramovLCI} and Andr\'e \cite{AndreHomologyCommutative}.
\end{remark}

We recall the following notion from \cite[Chapitre~0,~$\mathsection$~21]{EGAIV_I}.

\begin{definition}
    \label{def:p-basis}
    Let $R \to S$ be a map of $\mathbf{F}_p$-algebras and let $R[S^p]$ denote the image of the relative Frobenius $F_{S/R}$. Then $\{s_i\} \subset S$ is
    \begin{enumerate}
        \item \emph{$p$-independent over $R$} if the family of monomials 
        \[\left\{\prod s_i^{n_i} \colon 0 \leq n_i < p, n_i = 0 \text{ for all but finitely many } i \right\}\] is linearly independent over $R[S^p]$.
        \item  a \emph{$p$-basis of $S$ over $R$} if the family of monomials above is a basis of $S$ over $R[S^p]$.
    \end{enumerate}
\end{definition}

The following is an alternative characterization of a $p$-basis.

\begin{lemma}
    \label{lem:p-basis-another-characterization}
    Let $R \to S$ be a map of $\mathbf{F}_p$-algebras. Let $\{s_i\} \subset S$. Let $\{T_i\}_{i \in I}$ be a family of indeterminates indexed by $I$. The following are equivalent:
    \begin{enumerate}
        \item $\{s_i\}_{i \in I}$ is a $p$-basis of $S$ over $R$.
        \item The canonical $R[S^p]$-algebra map
        \begin{align*}
            S' \coloneqq \frac{R[S^p][T_i \colon i \in I]}{(T_i^p - s_i^p \colon i \in I)} &\longrightarrow S\\
            \overline{T_i} &\longmapsto s_i
        \end{align*}
        is an isomorphism.
    \end{enumerate}
\end{lemma}

\begin{proof}
    The family $\{\prod \overline{T_i}^{n_i} \colon 0 \leq n_i < p, n_i = 0 \text{ for all but finitely many } i\}$ is clearly a basis of $S'$ over $R[S^p]$. The result now follows by the definition of a $p$-basis because a linear map of free $R[S^p]$-modules that sends a basis to a basis has to be an isomorphism.
\end{proof}

\section{A characteristic \emph{p} analog of formally unramified, smooth and \'etale maps}

For an ideal $I$ of a ring $A$ and an integer $m \geq 0$, we will use $I^{[m]}$ to denote the ideal $(f^m \colon f \in I)$.

\begin{example}
   Let $R$ be a ring of prime characteristic $p > 0$ and $I$ an ideal. For any integer $e \geq 0$, if $F^e$ denotes the $e$-th iterate of the absolute Frobenius of $R$, then $F^e(I)R = I^{[p^e]}$.
\end{example}

\begin{definition}\cite[Project C]{Morrow2019THH} 
    \label{def:nil-ideal}
    An ideal $I$ of a ring $A$ is \emph{nil of bounded index} or \emph{bounded nil} (abbrv. \emph{b-nil}) if there exists $m \geq 0$ such that $I^{[m]} = 0$.
\end{definition}

Recall also that $I \subset A$ is \emph{nilpotent} if there is $m \geq 0$ so that $I^m = 0$, and $I$ is \emph{locally nilpotent} if every element of $I$ is nilpotent. 

\begin{example}
    If $I$ is a nilpotent ideal of a ring $A$, 
    then $I$ is bounded nil because $I^{[n]} \subset I^n$. A bounded nil ideal is clearly locally nilpotent. Furthermore a finitely generated bounded nil ideal is nilpotent. Indeed, if $I$ is generated by $n$ elements and $I^{[m]} = 0$, then $I^{nm} = 0$.
\end{example}

\begin{definition}\cite[Project C]{Morrow2019THH}
\label{def:b-nil formally-etale}
    A ring map $R \to S$ is called \emph{b-nil formally smooth} (resp. \emph{b-nil formally unramified}, resp. \emph{b-nil formally \'etale}) \footnote{Morrow uses the shorter term n-formally smooth (resp. unramified, resp. \'etale).} if for every bounded nil ideal $I$ of a ring $A$ and every solid commutative diagram 
    \begin{equation}
    \label{equn-thediagram}
    \begin{tikzcd}[cramped]
	{A/I} & S \\
	A & R
	\arrow[from=1-2, to=1-1]
	\arrow[color={rgb,255:red,214;green,92;blue,92}, dashed, from=1-2, to=2-1]
	\arrow[two heads, from=2-1, to=1-1]
	\arrow[from=2-2, to=1-2]
	\arrow[from=2-2, to=2-1]
\end{tikzcd}
    \end{equation}
    of rings, there exists at least one (resp. at most one, resp. exactly one) dashed arrow making the diagram commute. 
\end{definition}

\begin{lemma}
\label{lemma-standard}
Let $\varphi \colon R \to S$ be a ring map.
\begin{enumerate}[(1)]
    \item If $\varphi$ is b-nil formally smooth (resp. unramified, resp. \'etale), then $\varphi$ is formally smooth (resp. unramified, resp. \'etale).\label{lem:lemma-standard.1}
    
    \item A composition of  b-nil formally smooth (resp. unramified, resp. \'etale) maps is  b-nil formally smooth (resp. unramified, resp. \'etale).\label{lem:lemma-standard.2}
    
    \item A base change of a  b-nil formally smooth (resp. unramified, resp. \'etale) map is  b-nil formally smooth (resp. unramified, resp. \'etale).\label{lem:lemma-standard.3}
    
    \item If $R \to S$ and $S \to T$ are ring maps and both $R \to S$ and the composition $R \to S \to T$ are b-nil formally \'etale, then so is $S \to T$. \label{lem:lemma-standard.4}
    
    \item If $R \to S$ and $S \to T$ are ring maps such that the composition $R \to S \to T$ is b-nil formally unramified, then $S \to T$ is b-nil formally unramified.\label{lem:lemma-standard.5}
    
    \item A ring map $R \to S$ is  b-nil formally smooth (resp. unramified, resp. \'etale) if and only if for all surjective ring maps $A \to A/I$ which can be factored
    $$
    A \coloneqq \operatorname{lim}_n A_n \to \cdots \to A_{n+1} \to A_n \to \cdots \to A_0 = A/I
    $$
    in which each $A_{n+1} \twoheadrightarrow A_n$ is a surjective ring map whose kernel $I_{n+1}$ is bounded nil and all solid commutative diagrams \autoref{equn-thediagram}, there exists at least one (resp. at most one, resp. exactly one) dashed arrow $S\to A$ making the diagram \autoref{equn-thediagram} commute. \label{lem:lemma-standard.6}
    
    \item If $R \to S$ is a map of $\mathbf{F}_p$-algebras, then $R \to S$ is  b-nil formally smooth (resp. unramified, resp. \'etale) if and only if for all solid commutative diagrams \autoref{equn-thediagram} for pairs $(A,I)$ where $A$ is an $R$-algbera and $I$ is an ideal of $A$ such that $I^{[p]} = 0$, there exists at least one (resp. at most one, resp. exactly one) dashed arrow making the diagram commute. \label{lem:lemma-standard.7}
    
    \item If $\varphi \colon R \to S$ is a map of $\mathbf{F}_p$-algebras that is b-nil formally \'etale, then the relative Frobenius $F_\varphi$ is b-nil formally \'etale.\label{lem:lemma-standard.8}
    
    \item If $\varphi \colon R \to S$ is a map of $\mathbf{F}_p$-algebras that is b-nil formally unramified, then the relative Frobenius $F_\varphi$ is b-nil formally unramified. \label{lem:lemma-standard.9}
    
    \item If $R$ is an $\mathbf{F}_p$-algebra and $\mathbf{F}_p \to R$ is b-nil formally \'etale, then the absolute Frobenius $F_R \colon R \to R$ of $R$ is b-nil formally \'etale.\label{lem:lemma-standard.10}
    
    \item If $\varphi \colon R \to S$ is smooth, then $\varphi$ is b-nil formally smooth.\label{lem:lemma-standard.11}
    
    \item If $\varphi \colon R \to S$ is weakly \'etale, then $\varphi$ is b-nil formally \'etale.\label{lem:lemma-standard.12}
    
    \item If $\varphi \colon R \to S$ is \'etale, then $\varphi$ is b-nil formally \'etale.\label{lem:lemma-standard.13}
    
    \item If $\{X_\alpha \colon \alpha \in A\}$ is a family of indeterminates, then the polynomial extension $R \to R[X_\alpha \colon \alpha \in A]$ is b-nil formally smooth.\label{lem:lemma-standard.14}
    
    \item If $W \subset R$ is multiplicative, then $R \to W^{-1}R$ is b-nil formally \'etale.\label{lem:lemma-standard.15}
    
    \item A filtered colimit of b-nil formally \'etale $R$-algebras is a b-nil formally \'etale $R$-algebra.\label{lem:lemma-standard.16}
\end{enumerate}
\end{lemma}

\begin{proof}
We prove \autoref{lem:lemma-standard.4}, \autoref{lem:lemma-standard.7}, \autoref{lem:lemma-standard.8}, \autoref{lem:lemma-standard.10}, \autoref{lem:lemma-standard.11}, \autoref{lem:lemma-standard.12} and \autoref{lem:lemma-standard.13}. The rest of the properties are standard checks. We begin with \autoref{lem:lemma-standard.4}. Consider a solid diagram
\[
\begin{tikzcd}[cramped]
	{A/I} & T \\
	A & S
	\arrow[from=1-2, to=1-1]
	\arrow[two heads, from=2-1, to=1-1]
	\arrow[from=2-2, to=1-2]
	\arrow[from=2-2, to=2-1]
\end{tikzcd}
\]
where $I \subset A$ is a bounded nil ideal. Extending to the commutative diagram
\[\begin{tikzcd}[cramped]
	{A/I} & T \\
	A & S \\
	A & R
	\arrow[from=1-2, to=1-1]
	\arrow[two heads, from=2-1, to=1-1]
	\arrow[from=2-2, to=1-2]
	\arrow[from=2-2, to=2-1]
	\arrow["{\id_A}", from=3-1, to=2-1]
	\arrow[from=3-2, to=2-2]
	\arrow[from=3-2, to=3-1]
\end{tikzcd}\]
by the b-nil formally \'etale property of $R \to S \to T$, there exist a unique dashed arrow $T \to A$ such that the following diagram commutes:
\begin{equation}
\label{eq:expanded-square}
\begin{tikzcd}[cramped]
	{A/I} & T \\
	A & S \\
	A & R
	\arrow[from=1-2, to=1-1]
	\arrow["\exists!", color={rgb,255:red,214;green,92;blue,92}, dashed, from=1-2, to=3-1]
	\arrow[from=2-1, to=1-1]
	\arrow[from=2-2, to=1-2]
	\arrow["{\id_A}", from=3-1, to=2-1]
	\arrow[from=3-2, to=2-2]
	\arrow[from=3-2, to=3-1].
\end{tikzcd}
\end{equation}
Since the solid arrow $S \to A$ is the unique (because $R \to S$ is b-nil formally \'etale) one that makes 
\[\begin{tikzcd}[cramped]
	{A/I} \\
	A & S \\
	A & R
	\arrow[from=2-1, to=1-1]
	\arrow[from=2-2, to=1-1]
	\arrow[from=2-2, to=2-1]
	\arrow["{\id_A}", from=3-1, to=2-1]
	\arrow[from=3-2, to=2-2]
	\arrow[from=3-2, to=3-1]
\end{tikzcd}\]
commute (here $S \to A/I$ is the composition $S \to T \to A/I$), pre-composing the dashed arrow $T \to A$ with the map $S \to T$ must equal the solid arrow $S \to A$ by uniqueness, that is, the dashed arrow $T \to A$ makes the desired initial diagram 
\[\begin{tikzcd}[cramped]
	{A/I} & T \\
	A & S
	\arrow[from=1-2, to=1-1]
	\arrow[color={rgb,255:red,214;green,92;blue,92}, dashed, from=1-2, to=2-1]
	\arrow[from=2-1, to=1-1]
	\arrow[from=2-2, to=1-2]
	\arrow[from=2-2, to=2-1]
\end{tikzcd}\]
commute. The uniqueness of $T \to A$ for the last diagram follows from the uniqueness of $T \to A$ in \autoref{eq:expanded-square}.

    \autoref{lem:lemma-standard.7} If $I \subset A$ is an ideal in an $\mathbf{F}_p$-algebra such that there exists an integer $m \geq 0$ such that $I^{[m]} = 0$, then certainly we may take $m = p^e$ for some $e \gg 0$. Then use that $A \to A/I$ factors as
    $$
    A = A/I^{[p^e]} \to A/I^{[p^{e-1}]} \to \cdots \to A/I^{[p]} \to A/I
    $$
    in which each map is a quotient of an $\mathbf{F}_p$-algebra by an ideal $J$ satisfying $J^{[p]} = 0$ to get the desired unique lift. 

    \autoref{lem:lemma-standard.8} $F_\varphi$ makes the following diagram commute:
     \[
   \begin{tikzcd}[cramped]
	R && S \\
	{F_*R} && {F_*R\otimes_RS} \\
	&&& {F_*S}
	\arrow["\varphi", from=1-1, to=1-3]
	\arrow["{F_R}"', from=1-1, to=2-1]
	\arrow["{F_R\otimes_R\id_S}", from=1-3, to=2-3]
	\arrow["{F_S}", curve={height=-24pt}, from=1-3, to=3-4]
	\arrow["{\id_{F_*R}\otimes_R\varphi}"', from=2-1, to=2-3]
	\arrow["{F_*\varphi}"', curve={height=24pt}, from=2-1, to=3-4]
	\arrow["{\exists! F_\varphi}"{description}, color={rgb,255:red,214;green,92;blue,92}, from=2-3, to=3-4]
\end{tikzcd}
\]
Looking at the bottom triangle, $F_*\varphi$ is b-nil formally \'etale since it is just the map $\varphi$ while $\id_{F_*R}\otimes_R\varphi$ is b-nil formally \'etale by base change. Then $F_\varphi$ is b-nil formally \'etale by \autoref{lem:lemma-standard.4}.

\autoref{lem:lemma-standard.10} follows from \autoref{lem:lemma-standard.8} because the relative Frobenius of $\mathbf{F}_p \to R$ can be identified with the absolute Frobenius $F_R$ of $R$ because the absolute Frobenius of $\mathbf{F}_p$ is just the identity map.

\autoref{lem:lemma-standard.11} If $A$ is an $R$-algebra and $I \subset A$ is an ideal such that $I$ is bounded nil, then $I$ is locally nilpotent. Thus, the pair $(A,I)$ is Henselian by \cite[\href{https://stacks.math.columbia.edu/tag/0ALI}{Tag 0ALI}]{stacks-project}. A smooth map satisfies the lifting property with respect to all Henselian pairs by \cite[\href{https://stacks.math.columbia.edu/tag/0H74}{Tag 0H74}]{stacks-project}.

The proof of \autoref{lem:lemma-standard.12} is similar to \autoref{lem:lemma-standard.11} courtesy \cite[Theorem~1]{deJongOlander}. Namely, a map $R \to S$ is weakly \'etale if and only if it satisfies the unique lifting property for all Henselian pairs $(A,I)$, where $A$ is an $R$-algebra. As above, bounded nil ideals are locally nilpotent, and hence give rise to Henselian pairs.

Finally, \autoref{lem:lemma-standard.13} follows from \autoref{lem:lemma-standard.12} because an \'etale map is weakly \'etale by \cite[\href{https://stacks.math.columbia.edu/tag/092N}{Tag 092N}]{stacks-project}.
\end{proof}
%


\begin{remark}
    The converse of \autoref{lemma-standard}~\autoref{lem:lemma-standard.10} holds as well, that is, if $R$ is an $\mathbf{F}_p$-algebra such that the absolute Frobenius $F_R$ of $R$ is b-nil formally \'etale, then $\mathbf{F}_p \to R$ is b-nil formally \'etale. This follows from \autoref{thm:relative-Frobenius-iso-b-nil formally-etale} because $F_R$ can be identified with the relative Frobenius of $R/\mathbf{F}_p$.
\end{remark}

\subsection{B-nil formally smooth maps in prime characteristic}
Our goal in this section is to initiate a study of the b-nil formal smoothness property in prime characteristic. For example, we obtain an alternate characterization of b-nil formally smooth maps $R \to S$ of $\mathbf{F}_p$-algebras which demonstrates that such maps are close to satisfying that $S$ has a $p$-basis over $R$; see \autoref{thm:characterizing-b-nil-formal-smoothness} and \autoref{cor:p-basis-implies-bnil-formally-smooth}. The necessary and sufficient characterization of b-nil formal smoothness in \autoref{thm:characterizing-b-nil-formal-smoothness} is itself a consequence of the fact that for a b-nil formally smooth map $\varphi \colon R \to S$ of $\mathbf{F}_p$-algebras, the relative Frobenius $F_{\varphi}$ gives $S$ the structure of a projective $S^{(p)}$-module (\autoref{prop:b-nil-formally-smooth-injective-relative-Frobenius-projective}). 

To prove \autoref{prop:b-nil-formally-smooth-injective-relative-Frobenius-projective} and \autoref{thm:characterizing-b-nil-formal-smoothness} we need the following preparatory lemmas.

\begin{lemma}
\label{lemma-pthpowerofkernelofrelfrob}
    Let $\varphi \colon R \to S$ be a map of $\mathbf{F}_p$-algebras. Let $I \subset S^{(p)}$ be the kernel of $F_{\varphi}$. Then $I^{[p]} = 0$.
\end{lemma}

\begin{proof}
    Let $x = \sum r_i \otimes s_i \in I$. Then $F_{\varphi}(x) = \sum r_is_i^p  = 0$ and so
    \[
    x^p = \sum r_i^p \otimes s_i^p = \sum 1 \otimes r_is_i^p = 1 \otimes (\sum r_i s_i^p)= 0. \qedhere
    \]
\end{proof}

\begin{lemma}
\label{lem:dominating-bounded-nil}
    Let $R$ be a ring of prime characteristic $p>0$. Let $S$ be an $R$-algebra. There is a map of $R$-algebras $S' \to S$ with the following properties:
\begin{enumerate}
    \item $S' \to S$ is surjective.  \label{lem:dominating-bounded-nil.a}
    \item The kernel $J$ of $S' \to S$ is bounded nil. In fact $J^{[p]}=0$. \label{lem:dominating-bounded-nil.b}
    \item In the commutative diagram
\begin{equation}
\label{eq:rel-Frob-functorial}
\begin{tikzcd}
	{S'^{(p)}} & {S'} \\
	{S^{(p)}} & S
	\arrow["{F_{S'/R}}", from=1-1, to=1-2]
	\arrow[from=1-1, to=2-1]
	\arrow[from=1-2, to=2-2]
	\arrow["{F_{S/R}}", from=2-1, to=2-2]
\end{tikzcd}
\end{equation}
obtained from naturality of the relative Frobenius, the kernel of $F_{S'/R}$ is equal to the kernel of the left vertical map $S'^{(p)} \to S^{(p)}$.\label{lem:dominating-bounded-nil.c}
\end{enumerate}
Moreover, if $F_{S/R}$ is of finite presentation, then $S'$ can be chosen to be a finitely presented $S^{(p)}$-algebra such that $F_{S'/R}$ is also finitely presented and the kernel $J$ of $S' \twoheadrightarrow S$ is nilpotent.
\end{lemma}

We will repeatedly use that if $S$ is a finitely presented $R$-algebra and $T$ is a finitely generated $R$-algebra with an $R$-algebra surjection $T \twoheadrightarrow S$, then $\ker(T \twoheadrightarrow S)$ is finitely generated. This is a straightforward consequence of the fact that if one has maps $A \twoheadrightarrow B \to C$, then $\ker(A \to C)$ maps surjectively onto $\ker(B \to C)$. Thus, finite generation of $\ker(A \to C)$ implies that of $\ker(B \to C)$.

\begin{proof}
    Let $\{f_i\}_{i \in I}$ be a set of generators of $S$ as an $S^{(p)}$-algebra and take 
    \[S' = S^{(p)}[T_i \colon i \in I]/(T_i^p-(1\otimes f_i) \colon i \in I).\] 
    Let $S' \to S$ be given by the relative Frobenius $S^{(p)} \to S$ on $S^{(p)}$ and take $T_i \mapsto f_i$. Then $S' \to S$ is clearly surjective, so \autoref{lem:dominating-bounded-nil.a} holds. Consider the $R$-algebra $S[T_i: i \in I]/(T_i^p-f_i \colon i \in I)$. Its relative Frobenius is the map 
    $$
    S' = S^{(p)}[T_i: i \in I]/(T_i^p-(1 \otimes f_i) \colon i \in I) \to S[T_i: i \in I]/(T_i^p-f_i \colon i \in I)
    $$
    which is equal to $F_{S/R}$ on $S^{(p)}$ and takes $T_i$ to $T_i ^p = f_i$. This map has image contained in the subring $S \subset S[T_i: i \in I]/(T_i^p-f_i)$ and is identified with the surjection $S' \twoheadrightarrow S$. Thus, the kernel $J$ of $S' \twoheadrightarrow S$ is equal to the kernel of the relative Frobenius of $S[T_i]/(T_i^p - f_i)$, and so, $J^{[p]} = 0$ by \autoref{lemma-pthpowerofkernelofrelfrob}. This proves \autoref{lem:dominating-bounded-nil.b}.
   %
    That leaves \autoref{lem:dominating-bounded-nil.c}. Note that $S^{(p)}$ is an $R$-algebra via the map $R \to S^{(p)}$ that sends $R \ni r \mapsto r \otimes 1$. Thus, the relative Frobenius of $R \to S^{(p)}$ can be identified with the map $S^{(p^2)} \to S^{(p)}$ that sends $S^{(p^2)} \ni r \otimes s \mapsto r \otimes s^p$. We then have
    $$
    S'^{(p)} = S^{(p^2)}[T_i\colon i \in I]/(T_i^p-(1\otimes f_i) \colon i \in I)
    $$
    and the relative Frobenius $F_{S'/R} \colon S'^{(p)} \to S'$ of $R \to S'$ is the relative Frobenius $S^{(p^2)} \to S^{(p)}$ of $R \to S^{(p)}$ on $S^{(p^2)}$ and takes $T_i$ to $T_i^p =  1 \otimes f_i$ in $S'$. Thus, $S'^{(p)} \to S'$ (i.e. the top horizontal map in \autoref{eq:rel-Frob-functorial}) has image contained in the subring $S^{(p)}.$ On the other hand, the left vertical map $S'^{(p)} \to S^{(p)}$ in \autoref{eq:rel-Frob-functorial} is the base change of $S' \to S$ along the absolute Frobenius of $R$ and hence is equal to the relative Frobenius $S^{(p^2)} \to S^{(p)}$ of $R \to S^{(p)}$ on $S^{(p^2)}$ and sends $T_i \mapsto 1 \otimes f_i$. Thus, after restricting the co-domain of $F_{S'/R} \colon S'^{(p)} \to S'$ to $S^{(p)}$, the resulting map $S'^{(p)} \to S^{(p)}$ can then be identified with the left vertical arrow in \autoref{eq:rel-Frob-functorial}, so \autoref{lem:dominating-bounded-nil.c} follows. 

    Now assume $S$ is a finitely presented $S^{(p)}$-algebra. Then the set of generators of $S$ over $S^{(p)}$ can be chosen to be finite, that is, the index set $I$ above is finite. Then $S'$ is clearly a finitely presented $S^{(p)}$-algebra. Since $S' \twoheadrightarrow S$ is then a surjection of finitely presented $S^{(p)}$-algebras, it follows that $J = \ker(S' \twoheadrightarrow S)$ is a finitely generated ideal. But $J^{[p]} = 0$ by above, and a finitely generated bounded nil ideal is clearly nilpotent, that is, $J^n = 0$ for some $n \gg 0$. It remains to show in this case that $S'$ is also a finitely presented $S'^{(p)}$-algebra. Note that $R \to S^{(p)}$ is the base change of $R \to S$ along $F_R \colon R \to R$. Since the relative Frobenius is compatible with base change and since the base change of a finitely presented map is finitely presented, it follows that $F_{S^{(p)}/R}$ is finitely presented as well, that is, $S^{(p)}$ is a finitely presented $S^{(p^2)}$-algebra. Then $S' = S^{(p)}[T_i \colon i \in I]/(T_i^p - (1\otimes f_i) \colon i \in I)$ is a finitely presented $S^{(p^2)}$-algebra (note $I$ is a finite set). Since $S^{(p^2)} \to S'$ factors as $S^{(p^2)} \hookrightarrow S'^{(p)} \xrightarrow{F_{S'/R}} S'$ by the observation in the previous paragraph, and since $S'^{(p)}$ is a finitely presented $S^{(p^2)}$-algebra because $I$ is finite, it follows by \cite[\href{https://stacks.math.columbia.edu/tag/00F4}{Tag 00F4}(4)]{stacks-project} that $F_{S'/R}$ is finitely presented as well, i.e., $S'$ is a finitely presented $S'^{(p)}$-algebra.
\end{proof}

\begin{proposition}
\label{prop:b-nil-formally-smooth-injective-relative-Frobenius-projective}
    Let $\varphi: R \to S$ be a b-nil formally smooth map of $\mathbf{F}_p$-algebras. Then $S$ is a projective $S^{(p)}$-module via $F_\varphi$, and hence, $F_\varphi$ is faithfully flat and has an $S^{(p)}$-linear left inverse. In fact, if $\{f_i\}_{i \in I} \subset S$ is a set of generators of $S$ over $S^{(p)}$, then the canonical $S^{(p)}$-algebra surjection 
        \begin{align*}
            \frac{S^{(p)}[T_i \colon i \in I]}{(T_i^p - (1 \otimes f_i) \colon i \in I)} &\twoheadrightarrow S\\
            \overline{T_i} &\mapsto f_i
        \end{align*}
        has an $S^{(p)}$-algebra right inverse.
\end{proposition}

\begin{proof}
    We first show $F_\varphi$ is injective. Inductively build a sequence of $R$-algebra maps
    $$
    \cdots \to A_{n+1} \to A_n \to \cdots \to A_0 = S
    $$
    where each $A_{n+1} \to A_n$ satisfies conditions \autoref{lem:dominating-bounded-nil.a}-\autoref{lem:dominating-bounded-nil.c} of \autoref{lem:dominating-bounded-nil}. Since $R \to S$ is b-nil formally smooth, if $A \coloneqq \operatorname{lim}_n A_n$ then there is a map $S \to A$ of $R$-algebras which is a right inverse of $A \to A_0 = S$ (\autoref{lemma-standard}~\autoref{lem:lemma-standard.6}). We can think of the map $S \to A$ as a commutative diagram of maps
\[\begin{tikzcd}
	\cdots & {A_3} & {A_2} & {A_1} & {A_0 = S} \\
	&&&& S
	\arrow[from=1-1, to=1-2]
	\arrow[from=1-2, to=1-3]
	\arrow[from=1-3, to=1-4]
	\arrow[from=1-4, to=1-5]
	\arrow["\cdots"{description}, draw=none, from=2-5, to=1-1]
	\arrow[from=2-5, to=1-2]
	\arrow[from=2-5, to=1-3]
	\arrow[from=2-5, to=1-4]
	\arrow[from=2-5, to=1-5]
\end{tikzcd}\]
in the category of $R$-algebras. Base changing along the absolute Frobenius of $R$, we obtain a similar system where every term is replaced by its Frobenius twist, and hence a map $S^{(p)} \to \operatorname{lim}_nA_n^{(p)}$ which is right inverse to $\operatorname{lim}_nA_n^{(p)} \to A_0^{(p)} = S^{(p)}$ and fits into a diagram of $R$-algebras
\[\begin{tikzcd}
	S & \operatorname{lim}_n A_n & S \\
	{S^{(p)}} & {\operatorname{lim}_n A^{(p)}_n} & {S^{(p)}}
	\arrow[from=1-1, to=1-2]
	\arrow[from=1-2, to=1-3]
	\arrow["{F_{\varphi}}"', from=2-1, to=1-1]
	\arrow[from=2-1, to=2-2]
	\arrow[from=2-2, to=1-2]
	\arrow[from=2-2, to=2-3]
	\arrow["{F_{\varphi}}"', from=2-3, to=1-3]
\end{tikzcd}\]
where the horizontal arrows compose to the identity and the middle vertical arrow is given component-wise by relative Frobenius. It suffices to show the middle vertical arrow is injective. If $(x_n)_{n \geq 0} \in \operatorname{lim}_nA_n^{(p)}$ is an element which maps to zero, that is $F_{A_n/R}(x_n) = 0$ for all $n$, then by condition \autoref{lem:dominating-bounded-nil.c} of \autoref{lem:dominating-bounded-nil} we see that $x_{n-1} = 0$ for all $n\geq 1$ and so $(x_n)_{n \geq 0} = 0$, as needed. 

Let $\{f_i\}_{i \in I}$ be a set of generators of $S$ as an $S^{(p)}$-algebra. Consider the ring map $A = S^{(p)}[T_i \colon i \in I]/(T_i^{p}-(1 \otimes f_i) \colon i \in I) \to S$ which is $F_{\varphi}$ on $S^{(p)}$ and sends $T_i \to f_i$. This is a surjection of $R$-algebras with bounded nilpotent kernel, as shown in the proof of \autoref{lem:dominating-bounded-nil}. Since $R \to S$ is b-nil formally smooth, there is an $R$-algebra map $g: S \to A$ such that the composition $S \to A \to S$ is the identity.

\begin{claim}
\label{claim:retract}
The diagram
   \[\begin{tikzcd}
	S & S \\
	A = {S^{(p)}[T_i \colon i \in I]/(T_i^{p}-(1 \otimes f_i) \colon i \in I)} & {S^{(p)}}
	\arrow["{\operatorname{id}_S}"', from=1-2, to=1-1]
	\arrow[from=2-1, to=1-1]
    \arrow["g", from=1-2, to=2-1]
	\arrow["{F_{\varphi}}"', from=2-2, to=1-2]
	\arrow[from=2-2, to=2-1]
\end{tikzcd}\]
commutes.
\end{claim}

    The top triangle commutes by construction so the claim says that
    $g: S \to A$ is actually a map of $S^{(p)}$-algebras where $S$ is an $S^{(p)}$-algebra via $F_\varphi$. By the universal property of the tensor product it suffices to show that the composition $g \circ F_S$ is equal to the composition $S \to S^{(p)} \to A$ where $S \to S^{(p)}$ is $F_R \otimes_R \id_S$.
    We have that $\im(g \circ F_S)$ is contained in the subring of $A^p \subset A$, and thus  $\im(g \circ F_S) \subset S^{(p)}$. We may then view $g \circ F_S$ as a map $h: S \to S^{(p)}$ and since $A \to S$ is equal to $F_{\varphi}$ on the subring $S^{(p)}$, we see that $F_{\varphi} \circ h = F_S$ by the fact that the top triangle in the diagram commutes. However, $F_R \otimes_R \id_S \colon S \to S^{(p)}$ also satisfies that its post-composition with $F_{\varphi}$ is equal to $F_S$. Since we have seen that $F_{\varphi}$ is injective, we get $h = F_R \otimes_R \id_S$, and hence, $g \circ F_S = S \to S^{(p)} \to A$. The claim follows. 

    But now we see that $S$ is a direct summand as an $S^{(p)}$-module of the free $S^{(p)}$-module $A$, thus $S$ is a projective $S^{(p)}$-module. Since projective modules are flat and $F_\varphi$ is surjective on spectra, we get $F_\varphi$ is faithfully flat. Now the fact that $F_\varphi$ also splits as a map of $S^{(p)}$-modules follows by the following more general observation.
\end{proof}

\begin{lemma}
    \label{lem:projective-implies-split}
    Let $\varphi \colon R \to S$ be a ring map such that $S$ is a projective $R$-module and such that for all maximal ideals $\m$ of $R$, $\m S \neq S$ (for e.g., if $\Spec(S) \twoheadrightarrow \Spec(R)$). Then $\varphi$ has an $R$-linear left-inverse.
\end{lemma}

\begin{proof}
    The result follows by \cite[Remark~4.1.2~(f), Lemma~4.1.5~(2)]{DattaEpsteinTuckerMLmodules} but we reproduce the argument here for the reader's convenience. Our assumption implies $R \to S$ is faithfully flat. Let $S$ be a direct summand of a free $R$-module $F$ with basis $\{x_\alpha\}_{\alpha \in A}$. Let $\pi_\alpha \colon F \to R$ be projection onto $x_\alpha$. Let $a_\alpha$ be the coefficient of $x_\alpha$ when you express $1_S \in S$ in terms of the basis. Let $\Tr_S(1_S) \coloneqq \im(\Hom_R(S,R) \xrightarrow{\text{eval @ $1_S$}} R)$. Similarly, let $\Tr_F(1_S) \coloneqq \im(\Hom_R(F,R) \xrightarrow{\text{eval @ $1_S$}} R)$. Clearly, $1_S \in \Tr_F(1_S)F$ because for all $\alpha$, $a_\alpha \in \Tr_F(1_S)$. Now $S \hookrightarrow F$ implies by functoriarity that $\Tr_F(1_S) \subset \Tr_S(1_S)$. Since $S \hookrightarrow F$ is universally injective, we have $\Tr_F(1_S)S = \Tr_F(1_S)F \cap S$. Thus, $1_S \in \Tr_F(1_S)S \subset \Tr_S(1_S)S$. But $\Tr_S(1_S)$ is an ideal of $R$ and $R \to S$ is faithfully flat. So, $\Tr_S(1_S) = \Tr_S(1_S)S \cap R = R$. Since $1 \in \Tr_S(1_S)$, this precisely means there exists an $R$-linear map $S \to R$ which sends $1_S \mapsto 1_R$, that is, $\varphi$ splits.
\end{proof}

\begin{remark}
\label{rem:relative-Frob-b-nil-smooth-projective}
    If $\varphi \colon R \to S$ is a map of $\mathbf{F}_p$-algebras such that $F_\varphi$ is b-nil formally smooth\footnote{We will see in \autoref{thm:formally-etale-Noetherian-relative-Frobenius-iso} that this implies $F_\varphi$ is an isomorphism.}, then it is also the case that $S$ is a projective $S^{(p)}$-module. The proof in this case is shorter, so we provide it because we will use this fact in the proof of \autoref{thm:formally-etale-Noetherian-relative-Frobenius-iso}. As in the proof of \autoref{prop:b-nil-formally-smooth-injective-relative-Frobenius-projective}, let $\{f_i\}_{i \in I}$ be a set of generators of $S$ as an $S^{(p)}$-algebra. Then $1 \otimes f_i \in S^{(p)}$ maps to $f_i^p$ via $F_\varphi$. Consider the commutative square
\[\begin{tikzcd}
	S & S \\
	{S^{(p)}[T_i \colon i \in I]/(T_i^{p}- (1\otimes f_i) \colon i \in I)} & {S^{(p)}}
	\arrow["{\operatorname{id}_S}"', from=1-2, to=1-1]
	\arrow[from=2-1, to=1-1]
	\arrow["{F_{\varphi}}"', from=2-2, to=1-2]
	\arrow[from=2-2, to=2-1]
\end{tikzcd}\]
where the ring map $S^{(p)}[T_i \colon i \in I]/(T_i^{p}-(1\otimes f_i) \colon i \in I) \to S$ is $F_{\varphi}$ on $S^{(p)}$ and sends $T_i \to f_i$. We have seen this map has bounded nil kernel (see the proof of \autoref{lem:dominating-bounded-nil}~\autoref{lem:dominating-bounded-nil.b}). Now, since $F_{\varphi}$ is b-nil formally smooth there exists a ring map $S \to S^{(p)}[T_i \colon i \in I]/(T_i^{p}-(1\otimes f_i) \colon i \in I)$ making the above diagram commute. In particular, we see that $S$ is a direct summand of the free $S^{(p)}$-module $S^{(p)}[T_i \colon i \in I]/(T_i^{p}-(1\otimes f_i)\colon i \in I)$, which proves the claim.
\end{remark}

\begin{corollary}
\label{cor-geometricallyreduced}
Let $\varphi \colon R \to S$ be a nil-formally smooth map of $\mathbf{F}_p$-algebras. 
\begin{enumerate}
    \item If the absolute Frobenius of $R$ is universally injective (i.e. $R$ is $F$-pure), then $S$ is reduced.\label{cor-geometricallyreduced.a}

    \item If $R$ is a field, then $S$ is geometrically reduced over $R$.\label{cor-geometricallyreduced.b}
\end{enumerate}
\end{corollary}

\begin{proof}
\autoref{cor-geometricallyreduced.a} By \autoref{prop:b-nil-formally-smooth-injective-relative-Frobenius-projective}, the relative Frobenius $F_\varphi$ is injective. Since $F_R$ is universally injective, so is $F_R \otimes_R \id_S$. Then $F_S = F_\varphi \circ (F_R \otimes_R \id_S)$ is injective, and hence, $S$ is reduced.

\autoref{cor-geometricallyreduced.b} Since the property of being b-nil formally smooth is preserved under arbitrary base change, it suffices to show that if $R$ is a field and $R \to S$ is b-nil formally smooth, then $S$ is reduced. But this follows from \autoref{cor-geometricallyreduced.a} because fields have universally injective absolute Frobenius.
\end{proof}

We also obtain an analog of \autoref{prop:b-nil-formally-smooth-injective-relative-Frobenius-projective} for formal smoothness assuming the relative Frobenius is finitely presented.

\begin{proposition}
    \label{prop:formal-smoothness-fp-relFrob-injective-projective}
    Let $\varphi \colon R \to S$ be a map of $\mathbf{F}_p$-algebras such that $F_\varphi$ is a finitely presented ring map. Suppose $\varphi$ is formally smooth. Then $S$ is a projective $S^{(p)}$-module via $F_\varphi$, and so, $F_\varphi$ is faithfully flat and has an $S^{(p)}$-linear left inverse. In fact, if $f_1,\dots,f_n \in S$ generate $S$ as an $S^{(p)}$-algebra, then the canonical $S^{(p)}$-algebra surjection 
        \begin{align*}
            \frac{S^{(p)}[T_1,\dots,T_n]}{(T_1^p - (1 \otimes f_1),\dots, T_n^p - (1\otimes f_n))} &\twoheadrightarrow S\\
            \overline{T_i} &\mapsto f_i
        \end{align*}
        has an $S^{(p)}$-algebra right inverse.
\end{proposition}

\begin{proof}
    This is a straightforward adaptation of the argument given for b-nil formal smoothness in \autoref{prop:b-nil-formally-smooth-injective-relative-Frobenius-projective}. For the injectivity of $F_\varphi$ the only difference is that one uses the part of \autoref{lem:dominating-bounded-nil} for finitely presented relative Frobenius to get surjections $A_{n+1} \twoheadrightarrow A_n$ whose kernel is nilpotent and one then uses the lifting property of formal smoothness for surjections by nilpotent ideals. Next, the set of generators of $S$ over $S^{(p)}$ can be chosen to be finite, say $\{f_1,\dots,f_n\}$. Then the kernel of the $S^{(p)}$-algebra surjection 
    \begin{align*}
            \frac{S^{(p)}[T_1,\dots,T_n]}{(T_1^p - (1 \otimes f_1),\dots, T_n^p - (1\otimes f_n))} &\twoheadrightarrow S\\
            \overline{T_i} &\mapsto f_i
        \end{align*}
    is finitely generated since $S$ is a finitely presented $S^{(p)}$-algebra. This kernel is nilpotent by the proof of the finitely presented case of  \autoref{lem:dominating-bounded-nil}~\autoref{lem:dominating-bounded-nil.b}. The rest of the argument of \autoref{prop:b-nil-formally-smooth-injective-relative-Frobenius-projective} can now be repeated verbatim.
\end{proof}

\begin{remark}
    \label{rem:when-rel-Frob-fp}
    The relative Frobenius $F_\varphi$ of $\varphi \colon R \to S$ is finitely presented for instance when $F_R$ is finite and $F_S$ is finitely presented. Indeed, since $F_S = F_\varphi \circ (F_R \otimes_R \id_S)$ and $F_R \otimes_R \id_S$ is finite by base change, this is again an application of \cite[\href{https://stacks.math.columbia.edu/tag/00F4}{Tag 00F4}(4)]{stacks-project}.
\end{remark}


We now get the following necessary and sufficient characterization of b-nil formal smoothness.

\begin{theorem}
    \label{thm:characterizing-b-nil-formal-smoothness}
    Let $\varphi \colon R \to S$ be a map of $\mathbf{F}_p$-algebras. Then the following are equivalent:
    \begin{enumerate}
        \item\label{thm:characterizing-b-nil-formal-smoothness.a} $\varphi$ is b-nil formally smooth.
        \item\label{thm:characterizing-b-nil-formal-smoothness.b} There exists $\{f_i\}_{i \in I} \subset S$ such that the canonical $S^{(p)}$-algebra map 
        \[\frac{S^{(p)}[T_i \colon i \in I]}{(T_i^p - (1\otimes f_i) \colon i \in I)} \to S\] 
        that acts on $S^{(p)}$ via $F_\varphi$ and sends the class of $T_i$ to $f_i$ admits an $S^{(p)}$-algebra right-inverse. 
    \end{enumerate}
\end{theorem}

The proof of \autoref{thm:characterizing-b-nil-formal-smoothness} (as well as \autoref{thm:relative-Frobenius-iso-b-nil formally-etale}) will utilize the following lifting lemma.

\begin{lemma}
\label{lem:lifting-lemma-Frobenius}
    Suppose we are given a commutative diagram of rings of prime characteristic $p> 0$
\[\begin{tikzcd}[cramped]
	{A/I} & S \\
	A & R
	\arrow["f"', from=1-2, to=1-1]
	\arrow[two heads, from=2-1, to=1-1]
	\arrow["\varphi"', from=2-2, to=1-2]
	\arrow[ from=2-2, to=2-1]
\end{tikzcd}\]
in which $I^{[p]} = 0$. Then:
\begin{enumerate}
    \item There is a ring map $\tau \colon S^{(p)} \to A$ which makes the diagram
\[\begin{tikzcd}
	{A/I} & S & {S^{(p)}} \\
	A & R
	\arrow["f"',from=1-2, to=1-1]
	\arrow["{F_{\varphi}}"', from=1-3, to=1-2]
	\arrow[color={rgb,255:red,214;green,92;blue,92}, "\tau"', dashed, from=1-3, to=2-1]
	\arrow[two heads, from=2-1, to=1-1]
	\arrow[ from=2-2, to=1-2]
	\arrow[from=2-2, to=1-3]
	\arrow[from=2-2, to=2-1]
\end{tikzcd}\]
commute such that if $s \in S$ and $\alpha \in A$ is \emph{any} lift of $f(s)$, then $\tau(1\otimes s) = \alpha^p$.\label{lem:lifting-lemma-Frobenius.a}
\item Given maps $g, h: S \to A$ which both make the diagram 
\[\begin{tikzcd}
	{A/I} & S \\
	A & R
	\arrow[from=1-2, to=1-1]
	\arrow[from=1-2, to=2-1]
	\arrow[two heads, from=2-1, to=1-1]
	\arrow[from=2-2, to=1-2]
	\arrow[from=2-2, to=2-1]
\end{tikzcd}\]
commute, we have $g \circ F_{\varphi} = h \circ F_{\varphi}$.\label{lem:lifting-lemma-Frobenius.b}
\end{enumerate}
\end{lemma}

\begin{proof}
    For \autoref{lem:lifting-lemma-Frobenius.a}, by the universal property of the tensor product, giving a ring map $\tau \colon S^{(p)} \to A$ making the diagram commute is equivalent to giving an $R$-algebra map $\psi: S \to A$ such that $\psi(s) + I = f(s)^p$ for all $s \in S$. But note that $F_A : A \to A$ factors through the map 
    \begin{align*}
        k : A/I &\to A\\
        a + I &\mapsto a^p
    \end{align*}
    since $I^{[p]} = 0$. The composition $A/I \xrightarrow{k} A \twoheadrightarrow A/I$ is the absolute Frobenius on $A/I$. Define $\psi \coloneqq k \circ f$. This is an $R$-algebra map which by construction satisfies $\psi(s) + I = f(s)^p$. If $s \in S$ and $\alpha \in A$ such that $f(s) = \alpha + I$, then $\psi(s) = k(\alpha + I) = \alpha^p$, and so, $\tau(1\otimes s) = \psi(s) = \alpha^p$.

    For \autoref{lem:lifting-lemma-Frobenius.b}, showing that $g \circ F_\varphi = h \circ F_\varphi$ is equivalent to showing $g \circ F_S = h \circ F_S$ by the universal property of the tensor product. This is true because for $s \in S$ we have $g(s)-h(s) \in I$ and so $g(s^p) - h(s^p) = (g(s)-h(s))^p  = 0$ since $I^{[p]} = 0$.
\end{proof}

\begin{remark}
\label{rem:surjective-rel-Frobenius-b-nil-formally-unramified}
    \autoref{lem:lifting-lemma-Frobenius}~\autoref{lem:lifting-lemma-Frobenius.b} readily implies that if $F_\varphi$ is an epimorphism, then $\varphi$ is b-nil formally unramified hence formally unramified. The converse also holds and will be published in future work.
\end{remark}

\begin{proof}[Proof of \autoref{thm:characterizing-b-nil-formal-smoothness}]
    We have already shown \autoref{thm:characterizing-b-nil-formal-smoothness.a}$\implies$\autoref{thm:characterizing-b-nil-formal-smoothness.b} in \autoref{prop:b-nil-formally-smooth-injective-relative-Frobenius-projective}. We now show the converse. Consider a solid commutative diagram 
    \[\begin{tikzcd}[cramped]
	{A/I} & S \\
	A & R
	\arrow["\phi"', from=1-2, to=1-1]
	\arrow[two heads, from=2-1, to=1-1]
	\arrow[from=2-2, to=1-2]
	\arrow[from=2-2, to=2-1]
\end{tikzcd}\]
where $I \subset A$ is an ideal such that $I^{[p]} = 0$. By \autoref{lem:lifting-lemma-Frobenius}~\autoref{lem:lifting-lemma-Frobenius.a}, there is an $R$-algebra map $\tau \colon S^{(p)} \to A$ such that $\im(\tau) \subset A^p$ and such that
\begin{equation}
\label{eq:diagram}
\begin{tikzcd}
	{A/I} & S & {S^{(p)}} \\
	A & R
	\arrow["\phi"', from=1-2, to=1-1]
	\arrow["{F_{\varphi}}"', from=1-3, to=1-2]
	\arrow[color={rgb,255:red,214;green,92;blue,92}, "\tau"', dashed, from=1-3, to=2-1]
	\arrow[from=2-1, to=1-1]
	\arrow[ from=2-2, to=1-2]
	\arrow[from=2-2, to=1-3]
	\arrow[from=2-2, to=2-1]
\end{tikzcd}
\end{equation}
commutes.

Let $\beta \colon S^{(p)}[T_i \colon i \in I]/(T_i^p - (1\otimes f_i) \colon i \in I) \to S$ be the $S^{(p)}$-algebra map that sends the class of $T_i$ to $f_i$. Consider the diagram
\[\begin{tikzcd}[cramped]
	{A/I} & {\frac{S^{(p)}[T_i \colon i \in I]}{(T_i^p - (1\otimes f_i) \colon i \in I)}} \\
	A & {S^{(p)}}
	\arrow["{\phi \circ \beta}"', from=1-2, to=1-1]
	\arrow[two heads, from=2-1, to=1-1]
	\arrow[hook, from=2-2, to=1-2]
	\arrow["\tau", from=2-2, to=2-1],
\end{tikzcd}\]
which commutes because of \autoref{eq:diagram} and the fact that $\beta$ is an $S^{(p)}$-algebra map.
If $g_i \in A$ such that $g_i + I = \phi \circ \beta(\overline{T_i}) = \phi(f_i)$, then $\tau(1 \otimes f_i) = g_i^p$ by \autoref{lem:lifting-lemma-Frobenius}~\autoref{lem:lifting-lemma-Frobenius.a} again. Thus,
\begin{align*}
    \eta \colon \frac{S^{(p)}[T_i \colon i \in I]}{(T_i^p - (1\otimes f_i) \colon i \in I)} &\to A\\
    \overline{T_i}&\mapsto g_i\\
    S^{(p)}\ni \lambda &\mapsto \tau(\lambda)
\end{align*}
is an $S^{(p)}$-algebra lift of $\phi \circ \beta$ along $A \twoheadrightarrow A/I$. But by hypothesis, $\beta$ has an $S^{(p)}$-algebra right-inverse $\alpha \colon S \to S^{(p)}[T_i \colon i \in I]/(T_i^p - (1\otimes f_i) \colon i \in I)$. It is then clear that $\eta \circ \alpha \colon S \to A$ is an $R$-algebra lift of $\phi \colon S \to A/I$ along $A \twoheadrightarrow A/I$. 
\end{proof}

\autoref{thm:characterizing-b-nil-formal-smoothness} allows us to improve \cite[$0_{\text{IV}}$,~Th\'eor\`eme~21.2.7]{EGAIV_I}.

\begin{corollary}
    \label{cor:p-basis-implies-bnil-formally-smooth}
    Let $\varphi \colon R \to S$ be a map of $\mathbf{F}_p$-algebras. Suppose that the following conditions are satisfied:
    \begin{enumerate}[(i)]
        \item $F_\varphi$ is injective.
        \item $S$ has a $p$-basis over $R$.
    \end{enumerate}
    Then $\varphi$ is b-nil formally smooth.
\end{corollary}

Note that $F_\varphi$ being injective is a necessary condition for $\varphi$ to be b-nil formally smooth by \autoref{prop:b-nil-formally-smooth-injective-relative-Frobenius-projective}.

\begin{proof}
    By assumption, $S^{(p)} \cong \im(F_\varphi) = R[S^p]$. Let $\{f_i\}_{i \in I} \subset S$ be a $p$-basis over $R$. Note that $S^{(p)}\ni 1 \otimes f_i \mapsto f_i^p \in R[S^p]$ under the isomorphism. Thus, if $\{T_i \colon i \in I\}$ is a family of indeterminates, then the isomorphism $S^{(p)} \xrightarrow{\cong} R[S^p]$ induces an isomorphism 
    \begin{align*}
        \frac{S^{(p)}[T_i \colon i \in I]}{(T_i^p - (1\otimes f_i) \colon i \in I)} &\longrightarrow \frac{R[S^p][T_i \colon i \in I]}{(T_i^p - f_i^p \colon i \in I)}\\
        \overline{T_i} &\longmapsto \overline{T_i}.
    \end{align*}
    Since there is an $R[S^p]$-algebra isomorphism $R[S^p][T_i \colon i \in I]/(T_i^p - f_i^p \colon i \in I) \cong S$ by \autoref{lem:p-basis-another-characterization}, we get an $S^{(p)}$-algebra isomorphism $S^{(p)}[T_i \colon i \in I]/(T_i^p - (1\otimes f_i) \colon i \in I) \cong S$. Then $R \to S$ is b-nil formally smooth by \autoref{thm:characterizing-b-nil-formal-smoothness}.
\end{proof}

We also get the following large class of ring maps for which the b-nil formally smooth and formally smooth properties coincide. Although this result is a corollary of \autoref{thm:characterizing-b-nil-formal-smoothness}, we state is as a theorem.

\begin{theorem}
    \label{thm:b-nil-formally-smooth-fp-relFrob}
    Let $\varphi \colon R \to S$ be a map of $\mathbf{F}_p$-algebras such that $F_\varphi$ is finitely presented (for e.g., if $F_R$ is finite and $F_S$ is finitely presented). Then $\varphi$ is b-nil formally smooth if and only if $\varphi$ is formally smooth.
\end{theorem}

\begin{proof}
    The interesting implication again is that formal smoothness implies b-nil formal smoothness. Let $f_1,\dots,f_n \in S$ be algebra generators over $S^{(p)}$. The canonical $S^{(p)}$-algebra surjection 
        \begin{align*}
            \frac{S^{(p)}[T_1,\dots,T_n]}{(T_1^p - (1 \otimes f_1),\dots, T_n^p - (1\otimes f_n))} &\twoheadrightarrow S\\
            \overline{T_i} &\mapsto f_i
        \end{align*}
        has an $S^{(p)}$-algebra right inverse by \autoref{prop:formal-smoothness-fp-relFrob-injective-projective}. Then $\varphi$ is b-nil formally smooth by \autoref{thm:characterizing-b-nil-formal-smoothness}.
\end{proof}



\subsection{B-nil formally \'etale maps in prime characteristic} We can now prove our main result on b-nil formally \'etale maps.

\begin{theorem}
\label{thm:relative-Frobenius-iso-b-nil formally-etale}
    Let $\varphi \colon R \to S$ be a map of $\mathbf{F}_p$-algebras. The following are equivalent.
    \begin{enumerate}
        \item $\varphi$ is b-nil formally \'etale.\label{thm:relative-Frobenius-iso-b-nil formally-etale.a}
        \item The relative Frobenius $F_{\varphi}$ is b-nil formally \'etale.\label{thm:relative-Frobenius-iso-b-nil formally-etale.b}
        \item The relative Frobenius $F_{\varphi}$ is b-nil formally smooth.\label{thm:relative-Frobenius-iso-b-nil formally-etale.b'}
        \item The relative Frobenius $F_{\varphi}$ is an isomorphism.\label{thm:relative-Frobenius-iso-b-nil formally-etale.c}
    \end{enumerate} 
\end{theorem}

\begin{proof}
The implication \autoref{thm:relative-Frobenius-iso-b-nil formally-etale.a} $\implies$ \autoref{thm:relative-Frobenius-iso-b-nil formally-etale.b} is part \autoref{lem:lemma-standard.8} of \autoref{lemma-standard}. We now establish \autoref{thm:relative-Frobenius-iso-b-nil formally-etale.c} $\implies$ \autoref{thm:relative-Frobenius-iso-b-nil formally-etale.a}. By \autoref{lemma-standard}~\autoref{lem:lemma-standard.7} it suffices to show that for any  commutative diagram 
\[\begin{tikzcd}[cramped]
	{A/I} & S \\
	A & R
	\arrow["f"', from=1-2, to=1-1]
	\arrow[two heads, from=2-1, to=1-1]
	\arrow["\varphi"', from=2-2, to=1-2]
	\arrow["j", from=2-2, to=2-1],
\end{tikzcd}\]
where $I \subset A$ is an ideal such that $I^{[p]} = 0$, there exists a unique $R$-algebra lift $\widetilde{f} \colon S \to A$ of $f$ along $A \twoheadrightarrow A/I$. For existence, let $a: S^{(p)} \to A$ be the ring map found in \autoref{lem:lifting-lemma-Frobenius.a} of \autoref{lem:lifting-lemma-Frobenius}. Then $a \circ F_{\varphi}^{-1}$ works. For uniqueness, if $g, h$ are two $R$-algebra lifts of $f$ along $A \twoheadrightarrow A/I$, then by \autoref{lem:lifting-lemma-Frobenius.b} of \autoref{lem:lifting-lemma-Frobenius}, we have $g \circ F_\varphi = h \circ F_\varphi$ hence $g = h$ as $F_{\varphi}$ is surjective. 

Clearly \autoref{thm:relative-Frobenius-iso-b-nil formally-etale.b} $\implies$ \autoref{thm:relative-Frobenius-iso-b-nil formally-etale.b'}. It remains to show \autoref{thm:relative-Frobenius-iso-b-nil formally-etale.b'} $\implies$ \autoref{thm:relative-Frobenius-iso-b-nil formally-etale.c}. So assume $\varphi \colon R \to S$ has b-nil formally smooth relative Frobenius $F_{\varphi}$. We know that $S$ is a projective $S^{(p)}$-module via $F_{\varphi}$ by \autoref{rem:relative-Frob-b-nil-smooth-projective}, and $F_{\varphi}$ is a universal homeomorphism on spectra. Also, for any ring map $S^{(p)} \to k$ where $k$ is a perfect field, the base change $k \to S \otimes _{S^{(p)}} k$ of $F_{\varphi}$ to $k$ is b-nil formally smooth (\autoref{lemma-standard}\autoref{lem:lemma-standard.3}) and $\operatorname{Spec}(S \otimes _{S^{(p)}} k)$ is a one point space. By \autoref{cor-geometricallyreduced}, we conclude that $S \otimes _{S^{(p)}} k$ is a field since it is reduced, and it is a purely inseparable (algebraic) extension of $k$ by the fact that $F_{\varphi}$ is purely inseparable. Since $k$ is perfect we deduce that $k \to S \otimes _{S^{(p)}} k$ is an isomorphism. 

Now the result is formal. We have a ring map $B \to C$ such that $C$ is projective as a $B$-module and for every map $B \to k$ where $k$ is perfect, the base change $k \to C \otimes _B k$ is an isomorphism. We claim that it follows that $B \to C$ is an isomorphism. Indeed, we may assume $B$ is a local ring, in which case $C$ is a free $B$-module \cite[\href{https://stacks.math.columbia.edu/tag/0593}{Tag 0593}]{stacks-project} and we see that $C$ is in fact free of rank one by base changing to the algebraic closure of the residue field of $B$. Then $B \to C$ is a map of free rank one $B$-modules which induces an isomorphism after base change to the algebraic closure of the residue field of $B$, and so $B \to C$ is an isomorphism.
\end{proof}


Since b-nil formally \'etale maps are formally \'etale (\autoref{lemma-standard}~\autoref{lem:lemma-standard.1}), we then recover:

\begin{corollary}
    \label{cor:relative-Frobenius-iso-implies-formally-etale}
    Let $\varphi \colon R \to S$ be a homomorphism of rings of prime characteristic $p > 0$. If $F_\varphi$ is an isomorphism, then $\varphi$ is formally \'etale.
\end{corollary}

\begin{corollary}
    \label{cor:recovering-classical-formally-etale-field-extensions}
    Let $L/K$ be an extension of fields. Then:
    \begin{enumerate}
        \item $L/K$ is formally \'etale if and only if $L/K$ is b-nil formally \'etale.\label{cor:recovering-classical-formally-etale-field-extensions.a}
        \item \label{cor:recovering-classical-formally-etale-field-extensions.b} The following are equivalent:
        \begin{enumerate}[(i)]
            \item $L/K$ is b-nil formally smooth.\label{cor:recovering-classical-formally-etale-field-extensions.b.i}
            \item $L/K$ is formally smooth.\label{cor:recovering-classical-formally-etale-field-extensions.b.ii}
            \item $L$ is geometrically reduced over $K$ (i.e. it is a separable $K$-algebra).\label{cor:recovering-classical-formally-etale-field-extensions.b.iii}
        \end{enumerate}
    \end{enumerate}
\end{corollary}

\begin{proof}
When $\Char(K) = 0$, $L/K$ being formally \'etale is equivalent to $L/K$ being separable algebraic, hence ind-\'etale, hence ind-b-nil formally \'etale. But a filtered colimit of b-nil formally \'etale algebras is b-nil formally \'etale. In addition, any extension of fields of characteristic $0$ is separably generated, and hence, it is the composition of a purely transcendental extension which is b-nil formally smooth by \autoref{lemma-standard} (since it is a localization of a polynomial extension) and a separable algebraic extension which is b-nil formally \'etale by the previous discussion. Thus, any extension of characteristic $0$ fields is b-nil formally smooth since this property is preserved under composition. 

From now on we assume $\Char(K) = p > 0$. The formally \'etale property for field extensions is equivalent to the relative Frobenius being an isomorphism by \autoref{thm:formally-etale-fields}. Thus, \autoref{cor:recovering-classical-formally-etale-field-extensions.a} follows by \autoref{thm:relative-Frobenius-iso-b-nil formally-etale}.

    For \autoref{cor:recovering-classical-formally-etale-field-extensions.b}, clearly \autoref{cor:recovering-classical-formally-etale-field-extensions.b.i} $\implies$ \autoref{cor:recovering-classical-formally-etale-field-extensions.b.ii}. In addition, the equivalence of \autoref{cor:recovering-classical-formally-etale-field-extensions.b.ii} and \autoref{cor:recovering-classical-formally-etale-field-extensions.b.iii} is classical; see \cite[Theorem~26.9]{MatsumuraCommutativeRingTheory}. It remains to show \autoref{cor:recovering-classical-formally-etale-field-extensions.b.iii} $\implies$ \autoref{cor:recovering-classical-formally-etale-field-extensions.b.i}. Let $L$ be a geometrically reduced field extension of $K$ and let $B \subset L$ be a $p$-basis over $K$. Then $B$ is algebraically independent over $K$ and $L/K(B)$ is formally \'etale by \cite[Theorem~26.8]{MatsumuraCommutativeRingTheory}. Since purely transcendental extensions are b-nil formally smooth and formally \'etale field extensions are b-nil formally \'etale, we then get $L/K$ is b-nil formally smooth by composition.
\end{proof}

Another interesting consequence is for b-nil formally smooth maps to semi-perfect rings.

\begin{corollary}
    \label{cor:b-nil-smooth-to-semi-perfect}
    Let $\varphi \colon R \to S$ be a b-nil formally smooth map of $\mathbf{F}_p$-algebras where $S$ is semi-perfect. Then $\varphi$ is b-nil formally \'etale.
\end{corollary}

\begin{proof}
    Since $S$ is semi-perfect, the relative Frobenius $F_\varphi$ is surjective, and since $\varphi$ is b-nil formally smooth, $F_\varphi$ is injective by \autoref{prop:b-nil-formally-smooth-injective-relative-Frobenius-projective}.
\end{proof}

\begin{corollary}
    \label{cor:maps-perfect-algebras-formally-etale}
    If $\varphi \colon R \to S$ is a map of perfect $\mathbf{F}_p$-algebras, then $\varphi$ is b-nil formally \'etale.
\end{corollary}

\begin{proof}
    Indeed, the relative Frobenius of a map of perfect $\mathbf{F}_p$-algebras is  an isomorphism.
\end{proof}



\begin{examples}
    \label{eg:formally-etale-not-b-nil formally-etale}
    We give some examples of formally \'etale ring maps that are not b-nil formally \'etale and also an example of a b-nil formally \'etale map with non-trivial cotangent complex. We fix a prime $p > 0$ and only consider algebras over $\mathbf{F}_p$ in this example.
    \begin{enumerate}
        \item The first one is \cite{Bhatt_imperfect_trivial_cc}. Following a suggestion by Gabber, Bhatt constructs  a non-reduced $\mathbf{F}_p$-algebra $R$ such that the cotangent complex $L_{R/\mathbf{F}_p} \simeq 0$, and hence, $\mathbf{F}_p \to R$ is formally \'etale. However, $\mathbf{F}_p \to R$ cannot be b-nil formally \'etale since $R$ is not reduced (\autoref{cor-geometricallyreduced}). \label{eg:formally-etale-not-b-nil formally-etale.a}

        \item The next example is inspired by \cite[\href{https://stacks.math.columbia.edu/tag/060H}{060H}]{stacks-project}, which corrects \cite[0, 19.10.3(a)]{EGAIV}.  
        Suppose $I$ is an ideal of a ring $R$ such that $I = I^2$. Then \cite[\href{https://stacks.math.columbia.edu/tag/060H}{060H}]{stacks-project} shows that $R \to R/I$ is always formally \'etale. For example, if $R$ is a non-Noetherian valuation domain of rank $1$ and $I$ is the maximal ideal of $R$, then $I = I^2$, and so, $R \twoheadrightarrow R/I$ is formally \'etale but not flat. Now suppose $R$ is a ring of prime characteristic $p > 0$. Then the relative Frobenius of $R \twoheadrightarrow R/I$ can be identified with the surjection $R/I^{[p]} \twoheadrightarrow R/I$ with kernel $I/I^{[p]}$. Then $R \twoheadrightarrow R/I$ will not be b-nil formally \'etale by \autoref{thm:relative-Frobenius-iso-b-nil formally-etale} if $I \neq I^{[p]}$. In other words, we need to construct a ring $R$ of prime characteristic $p > 0$ and an idempotent ideal $I$ of $R$ such that $I^{[p]} \neq I$. Note that such an $I$ must necessarily be non-finitely generated. Let $\{x_n \colon n \in \mathbf{N}\}$ be indeterminates and consider the ring
     \[R \coloneqq \mathbf{F}_p[x_n\colon n \in \mathbf{N}]/(x_n^p, x_n - x_{n+1}x_{n+2} \colon n \in \mathbf{N}).\]
    Let $\overline{x_n}$ be the image of $x_n$ in $R$ and consider the ideal $I = (\overline{x_n} \colon n \in \mathbf{N})$ of $R$. Since $\overline{x_n} = \overline{x_{n+1}}\cdot\overline{x_{n+2}} \in I^2$, it follows that $I = I^2$. However, $I^{[p]} = 0$, so $I^{[p]} \neq I$.
    We can also make an example with $R$ semi-perfect and hence $F$-finite by taking
    \[R \coloneqq \mathbf{F}_p[x_n^{1/p^\infty} \colon n \in \mathbf{N}]/(x_n^p, x_n - x_{n+1}x_{n+2} \colon n \in \mathbf{N}).\]
    and $I = (\overline{x_n} \colon n \in \mathbf{N})$. \label{eg:formally-etale-not-b-nil formally-etale.b}

    \item The following example was suggested to us by Bhatt. Let $I$ be an ideal of a ring $R$ such that $I^{[p]} = I$. Then $R \twoheadrightarrow R/I$ is b-nil formally \'etale. Indeed, the relative Frobenius $R/I^{[p]} \twoheadrightarrow R/I$ of $R \to R/I$ is then just the identity. One can now construct a b-nil formally \'etale ring map with non-trivial cotangent complex. Let $(V,\m,\kappa)$ be a perfect valuation ring of characteristic $p > 0$ of rank $1$. Let $t \in \m$ be a non-zero element, that is, $t$ is a pseudo-uniformizer. Then $\m = (t^{1/p^\infty})$, and so, $\m^{[p]} = \m$. In particular, $\m^2 = \m$ as well. This shows that $V_0 \coloneqq V/(t) \twoheadrightarrow \kappa$ is b-nil formally \'etale by the above discussion. We claim $L_{\kappa/V_0}$ is not trivial. Our first claim is that $L_{\kappa/V} \simeq 0$. If $V_n \coloneqq V/(t^{1/p^n})$, then $\kappa = \colim_{n \in \mathbf{N}} V/(t^{1/p^n})$, and so, $L_{\kappa/V} = \colim_{n \in \mathbf{N}} L_{V_n/V}$ by \cite[\href{https://stacks.math.columbia.edu/tag/08S9}{Tag 08S9}]{stacks-project}. Since $t^{1/p^n}$ is a regular element of $V$, $L_{V_n/V} \simeq (t^{1/p^n})/(t^{2/p^n})[1]$ by \cite[\href{https://stacks.math.columbia.edu/tag/08SJ}{Tag 08SJ}]{stacks-project}. Thus, $H^{n}(L_{\kappa/V}) = 0$ for all $n \leq -2$. But $H^0(L_{\kappa/V}) = \Omega_{\kappa/V} = 0$ (since $V \twoheadrightarrow \kappa$) and $H^{-1}(L_{\kappa/V}) = \m/\m^2 = 0$ by \cite[\href{https://stacks.math.columbia.edu/tag/08RA}{Tag 08RA}]{stacks-project}, proving our claim. Now using the canonical distinguished triangle (\cite[\href{https://stacks.math.columbia.edu/tag/08QX}{Tag 08QX}]{stacks-project})
    \[L_{V_0/V} \otimes_{V_0}^L \kappa \to L_{\kappa/V} \to L_{\kappa/V_0} \to L_{V_0/V} \otimes_{V_0}^L \kappa[1]\]
    and the fact that $L_{V_0/V} = (t)/(t^2)[1]$ (\cite[\href{https://stacks.math.columbia.edu/tag/08SJ}{Tag 08SJ}]{stacks-project}), we see that $L_{\kappa/V_0} \not\simeq 0$. \label{eg:formally-etale-not-b-nil formally-etale.c}
    \end{enumerate}
\end{examples}

\subsection{Connections with enhancements of flat absolute Frobenius}
\label{sec:b-nil-formally-smooth-flat-Frobenius}
We end this section with a discussion of how the b-nil formal smoothness property interacts with certain enhancements of $\mathbf{F}_p$-algebras with flat absolute Frobenius. These notions were studied for modules over arbitrary rings by Ohm-Rush \cite{OhmRu-content} and Raynaud-Gruson \cite{rg71}, and were perhaps first utilized in the study of singularities in prime characteristic commutative algebra by Hochster-Huneke \cite{HochsterHunekeFRegularityTestElementsBaseChange} (see also \cite{KatzmanParameterTestIdealOfCMRings, KatzmanLyubeznikZhangOnDiscretenessAndRationality, SharpAnExcellentFPureRing, SharpBigTestElements, HochsterJeffriesintflatness, DattaEpsteinTuckerMLmodules, DESTate}). We state the main result, which generalizes \cite[Proposition~3.2.7]{DESTate}.

\begin{proposition}
    \label{prop:FORT-FIF-ascent-b-nil-formal-smoothness}
    Let $R \to S$ be a map of $\mathbf{F}_p$-algebras that is b-nil formally smooth.
    \begin{enumerate}
        \item If $R$ is Frobenius Ohm-Rush trace, then $S$ is Frobenius Ohm-Rush trace.
        \item If $R$ is $F$-intersection flat, then $S$ is $F$-intersection flat.
    \end{enumerate}
\end{proposition}

\begin{remark}
    A ring $R$ of characteristic $p$ is \emph{Frobenius Ohm-Rush trace} (resp. \emph{$F$-intersection flat}) if $F_*R$ is an Ohm-Rush trace (resp. $F$-intersection flat) $R$-module. The definitions of Ohm-Rush trace and intersection flat modules may be found in \cite{DattaEpsteinTuckerMLmodules, HochsterJeffriesintflatness}. Note that by \cite[Part II, Proposition~2.3.4]{rg71}, $R$ is Frobenius Ohm-Rush trace if and only if $F_*R$ is a flat and strictly Mittag-Leffler $R$-module. Also, $R$ is $F$-intersection flat if and only if $F_*R$ is a flat and Mittag-Leffler $R$-module \cite[Theorem~4.3.1]{DattaEpsteinTuckerMLmodules}. Since strictly Mittag-Leffler modules are Mittag-Leffler, (Frobenius) Ohm-Rush trace implies ($F$)-intersection flat. A projective module over an arbitrary ring is Ohm-Rush trace \cite{OhmRu-content} (see also \cite[Lemma~4.1.5]{DattaEpsteinTuckerMLmodules}), and hence is intersection flat. The proof of this last fact is a minor generalization of the proof of \autoref{lem:projective-implies-split}.
\end{remark}

\begin{proof}[Proof of \autoref{prop:FORT-FIF-ascent-b-nil-formal-smoothness}]
    Both the Ohm-Rush trace and intersection flat properties for ring maps are preserved under arbitrary base change; see \cite[Proposition~4.1.9]{DattaEpsteinTuckerMLmodules} for the former and \cite[Theorem~4.3.1]{DattaEpsteinTuckerMLmodules} for the latter. Thus, if $R$ is Frobenius Ohm-Rush trace (resp. $F$-intersection flat), then $S^{(p)}$ is an Ohm-Rush trace (resp. intersection flat) $S$-module by base change. By \autoref{prop:b-nil-formally-smooth-injective-relative-Frobenius-projective} and the fact that projective modules are Ohm-Rush trace, $F_*S$ is an Ohm-Rush trace thus an intersection flat $S^{(p)}$-module. Since the Ohm-Rush trace and intersection flat properties are preserved under compositions of rings maps (see \cite[Lemma~4.1.7]{DattaEpsteinTuckerMLmodules} and \cite[Proposition~5.7(a)]{HochsterJeffriesintflatness}), it follows that if $R$ is Frobenius Ohm-Rush trace (resp. $F$-intersection flat), then the absolute Frobenius of $S$, which is the composition $S \to S^{(p)} \to F_*S$, is Ohm-Rush trace (resp. intersection flat), i.e., $S$ is Frobenius Ohm-Rush trace (resp. $F$-intersection flat).
\end{proof}

\begin{corollary}
    \label{cor:b-nil-formal-smoothness-over-field}
    Let $R$ be a ring of prime characteristic $p > 0$ such that $F_*R$ is a projective $R$-module (for e.g., if $R$ is a localization of a polynomial ring over a field, or an $F$-finite regular ring). If $R \to S$ is a b-nil formally smooth ring map, then $S$ is Frobenius Ohm-Rush trace. In fact, $F_*S$ is a projective $S$-module.
\end{corollary}

\begin{proof}
    By hypothesis, $R$ is Frobenius Ohm-Rush trace. Thus, $S$ is Frobenius Ohm-Rush trace by \autoref{prop:FORT-FIF-ascent-b-nil-formal-smoothness}. In fact, $S^{(p)}$ is a projective $S$-module by base change. By hypothesis and \autoref{prop:b-nil-formally-smooth-injective-relative-Frobenius-projective}, $F_*S$ is a projective $S^{(p)}$-module. Since $S \to S^{(p)} \to F_*S$ is the absolute Frobenius of $S$, we get $F_*S$ is a projective $S$-module.
\end{proof}

\begin{example}
    \label{eg:regular-not-bil-nil-smooth}
    In the previous results, the key point is that for a b-nil formally smooth map $R \to S$ of $\mathbf{F}_p$-algebras, the relative Frobenius $F_{S/R}$ makes $F_*S$ a projective $S^{(p)}$-module. We now give an example to show that this projectivity property can fail easily even when $F_{S/R}$ is faithfully flat. Let $k$ be a field of characteristic $p > 0$ such that $[k^{1/p}\colon k] = \infty$. Let $t$ be an indeterminate and consider the local ring $R \coloneqq k[t]_{(t)}$. Let $\widehat{R}$ denote the $t$-adic completion of $R$. Then $R \to \widehat{R}$ is a regular map of Noetherian rings, hence $F_{\widehat{R}/R}$ is faithfully flat (\autoref{thm:Radu-Andre}~\autoref{thm:Radu-Andre.a}). Also, $R \to \widehat{R}$ is formally smooth in the $t\widehat{R}$-adic topology \cite[\href{https://stacks.math.columbia.edu/tag/07NQ}{Tag 07NQ}]{stacks-project}. However, $F_*\widehat{R}$ is \emph{not} a projective $\widehat{R}^{(p)}$-module. Consequently, $R \to \widehat{R}$ is \emph{not} b-nil formally smooth by \autoref{prop:b-nil-formally-smooth-injective-relative-Frobenius-projective}. Indeed, since $F_*R$ is a free $R$-module, by base change $\widehat{R}^{(p)}$ is $\widehat R$-free. Thus, if $F_*\widehat{R}$ is a projective, and hence a free, $\widehat{R}^{(p)}$-module (note $\widehat{R}^{(p)}$ is local), then $F_*\widehat{R}$ would be a free $\widehat{R}$-module. The rank of $F_*\widehat{R}$ over $\widehat{R}$ must be infinite because $k$, being a homomorphic image of $\widehat{R}$, would be $F$-finite otherwise. But $F_*\widehat{R}$ is a $t\widehat{R}$-adically complete $\widehat{R}$-module, and a non-finitely generated free $\widehat{R}$-module \emph{cannot} be $t\widehat{R}$-adically complete by the next lemma. A similar argument shows that $k \to \widehat{R}$ is not b-nil formally smooth.
\end{example}

\begin{lemma}
    \label{lem:free-modules-infinite-rank-complete-local}
    Let $(A,\m,\kappa)$ be a Noetherian complete local ring of Krull dimension $\geq 1$. Then a free $A$-module of infinite rank is never $\m$-adically complete.
\end{lemma}

\begin{proof}
    For an $A$-module $M$, we let $\widehat{M}$ denote the $\m$-adic completion of $M$. Let $I$ be an infinite set and $J \subset I$ be a countable infinite subset. Then we get a split exact sequence $0 \to A^{\oplus J} \to A^{\oplus I} \to A^{\oplus I\setminus J} \to 0$, so we get a commutative diagram of short exact sequences 
    \[\begin{tikzcd}[cramped]
	0 & {A^{\oplus J}} & {A^{\oplus I}} & {A^{\oplus I\setminus J}} & 0 \\
	0 & {\widehat{A^{\oplus J}}} & {\widehat{A^{\oplus I}}} & {\widehat{A^{\oplus I\setminus J}}} & 0
	\arrow[from=1-1, to=1-2]
	\arrow[from=1-2, to=1-3]
	\arrow[from=1-2, to=2-2]
	\arrow[from=1-3, to=1-4]
	\arrow[from=1-3, to=2-3]
	\arrow[from=1-4, to=1-5]
	\arrow[from=1-4, to=2-4]
	\arrow[from=2-1, to=2-2]
	\arrow[from=2-2, to=2-3]
	\arrow[from=2-3, to=2-4]
	\arrow[from=2-4, to=2-5].
\end{tikzcd}\]
Since free $A$-modules are $\m$-adically separated, if $A^{\oplus I} \to \widehat{A^{\oplus I}}$ is an isomorphism, then so is $A^{\oplus J} \to \widehat{A^{\oplus J}}$ by the snake lemma. Thus, replacing $I$ by $J$, it suffices to that $A^{\oplus {\mathbf N}}$ is not $\m$-adically complete. Since $\dim(A) \geq 1$, for all $n \in \mathbf{N}$ there exists  $x_n \in \m^n \setminus \m^{n+1}$. Now consider the sequence $(y_n)_n \subset A^{\mathbf{N}}$, where $y_n = (x_0,\dots,x_n,0,0,\dots,0,\dots)$. Then $(y_n)_n$ is Cauchy in the $\m$-adic topology on $A^{\oplus \mathbf{N}}$ but there is no $\alpha \in A^{\oplus \mathbf{N}}$ such that $(y_n-\alpha)_n$ is a null sequence in the $\m$-adic topology, showing that $A^{\oplus \mathbf{N}}$ cannot be $\m$-adically complete. Indeed, if $\alpha = (a_0,a_1,\dots)$, then choose an integer $M \gg 0$ such that for all $n \geq M$, $a_n = 0$. Since $x_{M} \in \m^{M}\setminus\m^{M+1}$, for all $n \geq M$, $y_n - \alpha \notin \m^{M+1}(A^{\oplus \mathbf{N}})$. 
\end{proof}


%
%

\autoref{eg:regular-not-bil-nil-smooth} illustrates that $F_{S/R}$ can fail to give $S$ the structure of a projective $S^{(p)}$-module even for nice maps $R \to S$. That of course implies that $R \to S$ cannot be b-nil formally smooth. So one might naturally wonder if \autoref{eg:regular-not-bil-nil-smooth} is a formally smooth map. This also isn't the case, as the following direct argument shows.

\begin{remark}
\label{rem:infinite-field-char-p-not-formally-smooth}
    Let $R = k[t]_{(t)}$ and $\widehat{R} = k\llbracket t\rrbracket$, for a non-$F$-finite field of characteristic $p > 0$ as in \autoref{eg:regular-not-bil-nil-smooth}. Note that $R \to \widehat{R}$ is also not formally smooth. Indeed, this follows because $\Omega_{\widehat{R}/R}$ is not a projective $\widehat{R}$-module, see \cite[\href{https://stacks.math.columbia.edu/tag/031J}{Tag 031J}]{stacks-project}. Assume for contradiction that  $\Omega_{\widehat{R}/R}$ is $\widehat{R}$-projective, and hence, $\widehat{R}$-free. Since the formation of $\Omega$ commutes with base change, we have $\Omega_{\widehat{R}/R} \otimes_R k \cong \Omega_{\widehat{R} \otimes_R k/k} = 0$, and so, $\Omega_{\widehat{R}/R} = 0$. Thus, it is enough to show that $\Omega_{\widehat{R}/R} \neq 0$, that is, $R \to \widehat{R}$ is not formally unramified. We will show the localization $R[1/t] \to \widehat{R}[1/t]$ is not formally unramified. Note $R[1/t] = k(t)$ and $\widehat{R}[1/t] = k((t))$, and now recall that a field extension $L/K$ of characteristic $p > 0$ is formally unramified if and only if $L/K$ has an empty $p$-basis \cite[\href{https://stacks.math.columbia.edu/tag/07P2}{Tag 07P2}]{stacks-project}, that is, $L/K$ is formally unramified if and only if $L = KL^p = K[L^p] = L^p[K]$. But for $L = k((t))$ and $K = k(t)$, $KL^p \neq L$ precisely because $k$ is not $F$-finite. Indeed, one can show that the set $L' \subset L$ of Laurent series $\sum _{i \geq - N} a_i t^i$ such that $[k^p(a_{-N}, a_{-N+1}, \dots ) : k^p] < \infty$ is a subfield of $L$ containing both $k[T]$ and $L^p$, and therefore $KL^p \subset L'$. Since $[k : k^p] = \infty$, we have $KL^p \subset L' \subsetneq L$ as needed.
\end{remark}

We will systematically study the (b-nil) formal smoothness property for ideal adic completions of rings with respect to finitely generated ideals in \autoref{sec:formal-smooth-completion}.

\section{Comparing b-nil formally \'etale and formally \'etale maps}
\label{sec:Section-3}

Let $\varphi \colon R \to S$ be a map of $\mathbf{F}_p$-algebras. We have seen in the previous section that if the relative Frobenius $F_\varphi$ is an isomorphism, then $\varphi$ is formally \'etale (in fact, $\varphi$ satisfies the stronger property of being b-nil formally \'etale). Conversely, if $\varphi \colon R \to S$ is formally \'etale, then $F_\varphi$ is known to be an isomorphism in the following cases (in increasing levels of generality):
\begin{enumerate}[(1)]
    \item $\varphi$ is of finite presentation, i.e., $\varphi$ is \'etale; see \cite[Expos\'e~XV, $n^0~2$, Proposition~2~c)]{SGA5}.
    \item $\varphi$ is a filtered colimit of \'etale maps, i.e., $\varphi$ is ind-\'etale; see (3) below and \cite[\href{https://stacks.math.columbia.edu/tag/092N}{Tag 092N}]{stacks-project}. 
    \item $\varphi$ is \emph{weakly \'etale}, i.e., $\varphi$ and the induced multiplication map $S \otimes_R S \to S$ are both flat; see \cite[\href{https://stacks.math.columbia.edu/tag/0F6W}{Tag 0F6W}]{stacks-project}. Alternatively, one can use \autoref{thm:relative-Frobenius-iso-b-nil formally-etale} because a weakly \'etale ring map is b-nil formally \'etale by \cite[Theorem~1]{deJongOlander}. Indeed, a weakly \'etale ring map satisfies the unique lifting property with respect to all Henselian pairs $(A,I)$, and if $I$ is a bounded nil ideal of $A$, then $(A,I)$ is a Henselian pair by \cite[\href{https://stacks.math.columbia.edu/tag/0ALI}{Tag 0ALI}]{stacks-project}.
\end{enumerate}

In this section, we prove some cases of formally \'etale maps with isomorphic relative Frobenius that do not follow from the above ones. Our first result is about maps of rings with finite Frobenius, i.e., $F$-finite rings.

\begin{theorem}
    \label{thm:formally-etale-F-finite}
    Let $\varphi \colon R \to S$ be a formally \'etale map of rings of prime characteristic $p > 0$ such that $F_\varphi$ is a finitely presented ring map (for e.g., if $F_R$ is finite and $F_S$ is finitely presented). Then $F_\varphi$ is an isomorphism.
\end{theorem}

\begin{proof}
    Using the diagram
   \[
   \begin{tikzcd}[cramped]
	R && S \\
	{F_*R} && {F_*R\otimes_RS} \\
	&&& {F_*S}
	\arrow["\varphi", from=1-1, to=1-3]
	\arrow["{F_R}"', from=1-1, to=2-1]
	\arrow["{F_R\otimes_R\id_S}", from=1-3, to=2-3]
	\arrow["{F_S}", curve={height=-24pt}, from=1-3, to=3-4]
	\arrow["{\id_{F_*R}\otimes_R\varphi}"', from=2-1, to=2-3]
	\arrow["{F_*\varphi}"', curve={height=24pt}, from=2-1, to=3-4]
	\arrow["{\exists! F_\varphi}"{description}, color={rgb,255:red,214;green,92;blue,92}, from=2-3, to=3-4]
\end{tikzcd}
\]
    we see that since $F_*\varphi  = F_\varphi \circ (\id_{F_*R}\otimes \varphi)$ is formally \'etale and $\id_{F_*R}\otimes\varphi$ is formally \'etale by base change, $F_\varphi$ is formally \'etale as well (the proof is very similar to that of \autoref{lemma-standard}~\autoref{lem:lemma-standard.8}). Thus, $F_\varphi$ is \'etale since it is finitely presented. Since $\Spec(F_\varphi)$ is a universal homeomorphism, it is universally injective. Thus, $\Spec(F_\varphi)$ is an open immersion by \cite[\href{https://stacks.math.columbia.edu/tag/025G}{Tag 025G}]{stacks-project}, and hence an isomorphism of affine schemes, that is, $F_\varphi$ is an isomorphism of rings.
\end{proof}

\begin{remark}
\label{rem:following-bnil-formally-etale-Ffinite}
\autoref{eg:formally-etale-not-b-nil formally-etale}~\autoref{eg:formally-etale-not-b-nil formally-etale.b} shows that in \autoref{thm:formally-etale-F-finite} one cannot relax the assumption $F_\varphi$ is finitely presented to $F_\varphi$ is finite type (equivalently, $F_\varphi$ is finite). Namely, we construct a non-Noetherian semi-perfect ring $R$ of characteristic $p > 0$ and an ideal $I$ of $R$ such that $I = I^2$ but $I \neq I^{[p]}$. Then $R \to R/I$ is formally \'etale (hence also formally smooth) but $R \to R/I$ is not b-nil formally smooth and hence not b-nil formally \'etale. Indeed, a necessary condition for a map of $\mathbf{F}_p$-algebras is for its relative Frobenius to be injective (\autoref{prop:b-nil-formally-smooth-injective-relative-Frobenius-projective}). However, the relative Frobenius of $R \to R/I$ can be identified with the canonical map $R/I^{[p]} \twoheadrightarrow R/I$, which by our construction of $I$ is surjective but not injective. Since $R$ is semi-perfect so is $R/I$. The point here is that $F_{(R/I)/R}$ is finite (equivalently, of finite type) but not finitely presented.
\end{remark}

Our final main result is about formally \'etale maps of Noetherian rings; see \autoref{thm:formally-etale-Noetherian-relative-Frobenius-iso}. The proof builds on the following observations.

\begin{lemma}
    \label{lem:cyclically-pure-birational-extension}
    Let $R \hookrightarrow S$ be an extension of rings such that $S \subset Q(R)$, the total ring of fractions of $R$. Suppose that for all principal ideals $I$ of $R$, $IS \cap R = I$ (for e.g., if $R \hookrightarrow S$ is universally injective or faithfully flat). Then $R = S$.
\end{lemma}

\begin{proof}
    Since $S \subset Q(R)$, every element $x \in S$ can be expressed as fraction $a/b$ where $a, b \in R$ and $b$ is a non-zerodivisor. Since $a = xb$ in $S$, we then get $(aR)S = aS \subset bS = (bR)S$, and so, $aR = (aR)S \cap R \subset (bR)S \cap R = bR$. Hence, $b|a$, and so, $x = a/b \in R$.
\end{proof}

\begin{lemma}
    \label{lem:formally-unramified-field-extension-trivial}
    Let $L/K$ be a formally unramified extension of fields of characteristic $p > 0$ such that $L^p \subset K \subset L$. Then $L = K$.
\end{lemma}

\begin{proof}
    By assumption, the compositum $KL^p = K$. If $\Sigma$ is a $p$-basis of $L/K$, then $L = KL^p(\Sigma) = K(\Sigma)$ and $\Omega_{L/K}$ is an $L$-vector space with a basis in bijection with $\Sigma$; see for instance \cite[\href{https://stacks.math.columbia.edu/tag/07P2}{Tag 07P2}]{stacks-project}. Since $L/K$ is formally unramified, $\Omega_{L/K} = 0$, and so, $\Sigma = \emptyset$. Then $L = K(\Sigma) = K$.
\end{proof}

\begin{proposition}
    \label{prop:formally-etale-Noetherian-reduced}
    Let $\varphi \colon R \to S$ be a formally \'etale map of Noetherian rings of prime characteristic $p > 0$. Suppose $R$ or $S$ is reduced. Then $F_\varphi$ is an isomorphism.
\end{proposition}

\begin{proof}
    Since formally \'etale maps are formally smooth and both $R$ and $S$ are Noetherian, $\varphi$ is a flat map with geometrically regular fibers and $F_\varphi$ is a faithfully flat map of Noetherian rings by \autoref{thm:Radu-Andre}\autoref{thm:Radu-Andre.c}. Note that this implies that if $R$ is reduced, then so is $S$ by \cite[\href{https://stacks.math.columbia.edu/tag/0C21}{Tag 0C21}]{stacks-project}. Hence, it suffices to show the theorem assuming only that $S$ is reduced. 
    
    As we have seen in the proof of \autoref{thm:formally-etale-F-finite}, $F_\varphi$ is also formally \'etale. We claim that this implies $F_\varphi$ is an isomorphism. Let $\im(F_\varphi) \coloneqq R[S^p]$. By the injectivity of $F_\varphi$, it is enough to show $R[S^p] = S$. We have a faithfully flat formally \'etale and purely inseparable map of reduced Noetherian rings $R[S^p] \hookrightarrow S$. By \autoref{lem:cyclically-pure-birational-extension}, it suffices to show that the total ring of fractions $Q(R[S^p])$ equals the total ring of fractions $Q(S)$. Let $\frp_1,\dots,\frp_n \in \Spec(S)$ be the distinct minimal primes and let $\frq_i \coloneqq R[S^p] \cap \frp_i$. Note $\frq_1,\dots,\frq_n$ are the distinct minimal primes of $R[S^p]$ (one can use that $\Spec(F_\varphi)$ is a homeomorphism to see this). Since $Q(R[S^p]) \cong \prod_{i=1}^n R[S^p]_{\frq_i}$ and $Q(S) \cong \prod_{i=1}^n S_{\frp_i}$ (\cite[\href{https://stacks.math.columbia.edu/tag/02LX}{Tag 02LX}]{stacks-project}), it suffices to show that the extension of fields $R[S^p]_{\frq_i} \hookrightarrow S_{\frp_i}$ is an isomorphism for all $i$. Now, $R[S^p]_{\frq_i} \hookrightarrow S_{\frp_i}$ is a formally \'etale extension of fields by base change of the formally \'etale map $R[S^p] \hookrightarrow S$ and a further localization. Moreover, $(S_{\frp_i})^p \subset R[S^p]_{\frq_i}$ since the composition $S \hookrightarrow R[S^p] \hookrightarrow S$ is the absolute Frobenius on $S$ (the first map raises every element to the $p$-th power). Thus, we have a formally \'etale (hence formally unramified) extension of fields $S_{\frp_i}/R[S^p]_{\frq_i}$ such that $(S_{\frp_i})^p \subset R[S^p]_{\frp_i}$. Then $R[S^p]_{\frp_i} \cong S_{\frq_i}$ by  \autoref{lem:formally-unramified-field-extension-trivial}.
\end{proof}

As a consequence of the previous proposition, we now get unconditionally that:

\begin{theorem}
    \label{thm:formally-etale-Noetherian-relative-Frobenius-iso}
    Let $\varphi \colon R \to S$ be a map of Noetherian rings of prime characteristic $p > 0$. The following are equivalent:
    \begin{enumerate}
        \item $\varphi$ is formally \'etale.\label{thm:formally-etale-Noetherian-relative-Frobenius-iso.a}
        \item $F_\varphi$ is an isomorphism.\label{thm:formally-etale-Noetherian-relative-Frobenius-iso.b}
        \item $\varphi$ is b-nil formally \'etale.\label{thm:formally-etale-Noetherian-relative-Frobenius-iso.c}
    \end{enumerate}
\end{theorem}

\begin{proof}
By \autoref{thm:relative-Frobenius-iso-b-nil formally-etale} and \autoref{cor:relative-Frobenius-iso-implies-formally-etale}, it suffices to show \autoref{thm:formally-etale-Noetherian-relative-Frobenius-iso.a} $\implies$ \autoref{thm:formally-etale-Noetherian-relative-Frobenius-iso.b}. As in the proof of \autoref{prop:formally-etale-Noetherian-reduced}, we see that $\varphi$ is a regular map. If $I$ is the nilradical of $R$, then $I$ is a nilpotent ideal since $R$ is Noetherian. By base change, the induced map $\overline{\varphi} \colon R_{\red} \to S/IS$ is formally \'etale and regular. Note that this implies $S/IS$ is reduced by \cite[\href{https://stacks.math.columbia.edu/tag/0C21}{Tag 0C21}]{stacks-project} again, and so, $IS$ must be the nilradical of $S$ (since certainly $IS$ is contained in the nilradical). Thus, we have a co-cartersian square
\[\begin{tikzcd}[cramped]
	R & S \\
	{R_{\red}} & {S_{\red}}
	\arrow["\varphi", from=1-1, to=1-2]
	\arrow[two heads, from=1-1, to=2-1]
	\arrow[two heads, from=1-2, to=2-2]
	\arrow["{\varphi_{\red}}"', from=2-1, to=2-2],
\end{tikzcd}\]
where $\varphi_{\red}$ is a formally \'etale map of reduced Noetherian rings. Hence, \autoref{prop:formally-etale-Noetherian-reduced} implies that $F_{\varphi_{\red}}$ is an isomorphism. However, the formation of relative Frobenius commutes with base change. 
Thus, $\id_{R_{\red}} \otimes_R F_{\varphi} = F_{\varphi_{\red}}$, that is, after killing a nilpotent ideal of $R$, the map $F_{\varphi_{\red}}$, induced by the $R$-algebra map $F_{\varphi}$, becomes surjective. Then by the nilpotent version of Nakayama's lemma \cite[\href{https://stacks.math.columbia.edu/tag/00DV}{Tag 00DV~(11)}]{stacks-project}, $F_{\varphi}$ is surjective. But $F_{\varphi}$ is faithfully flat since $R \to S$ is a regular map of Noetherian rings, and so, $F_{\varphi}$ is injective.
\end{proof}

\begin{remarks}
\label{rem:concluding-remarks-main-thm-Noetherian}
{\*}
\begin{enumerate}
    \item Since a formally \'etale map of Noetherian rings is automatically flat (\autoref{thm:Radu-Andre}~\autoref{thm:Radu-Andre.c}), \autoref{thm:formally-etale-Noetherian-relative-Frobenius-iso} shows that a formally \'etale map of Noetherian $\mathbf{F}_p$-algebras is \emph{relatively perfect} in the sense of \cite[Definition~3.1]{FinkRelativelyPerfect} (see \cite[Remark~3.4]{FinkRelativelyPerfect}), thereby improving \cite[Theorem~5.3]{FinkRelativelyPerfect}.

    \item If $\varphi \colon R \to S$ is a formally unramified map of $\mathbf{F}_p$-algebras, then the relative Frobenius $F_{\varphi}$ is also formally unramified because $F_*\varphi = F_{\varphi} \circ (\id_{F_*R} \otimes_R \varphi)$ is formally unramified. One can now check that the proofs of \autoref{prop:formally-etale-Noetherian-reduced} and \autoref{thm:formally-etale-Noetherian-relative-Frobenius-iso} show that a regular map of Noetherian $\mathbf{F}_p$-algebras that is formally unramified has isomorphic relative Frobenius, that is, it is b-nil formally \'etale.

    \item Given the results of Radu, Andr\'e and Grothendieck (see \autoref{thm:Radu-Andre}), the main ingredient in \autoref{thm:formally-etale-Noetherian-relative-Frobenius-iso} is the surjectivity of $F_\varphi$ assuming that $\varphi \colon R \to S$ is a formally \'etale map of $\mathbf{F}_p$-algebras. A different proof of surjectivity can be given than the one above using the following non-trivial result: Suppose $A \to B$ is a map of Noetherian $\mathbf{F}_p$-algebras. Assume there exists a ring map $B \to A$ such that $A \to B \to A = F^e_A$ and $B \to A \to B = F^e_B$, i.e., $A \to B$ is invertible up to a power of Frobenius. Then $A \to B$ is formally unramified if and only if $A \to B$ is surjective. An accessible proof of this result appears in \cite[Theorem~11.2]{MaPolstraFSing}, who attribute it to Andr\'e. Coming back to our situation of a formally \'etale map $\varphi \colon R \to S$, we know $S^{(p)}$ is a Noetherian ring by \autoref{thm:Radu-Andre}~\autoref{thm:Radu-Andre.c}. Since $F_\varphi$ is formally \'etale if $\varphi$ is, $F_\varphi$ is a formally unramified map of Noetherian $\mathbf{F}_p$-algebras. But the canonical map $F_R \otimes_R \id_S \colon S \to S^{(p)}$ is an inverse of $F_{\varphi}$ up to Frobenius. Then $F_{\varphi}$ is surjective by the Andr\'e-Ma-Polstra result. Note, however, the proof of this latter result is more technical than the argument for surjectivity we have given above.
\end{enumerate}
\end{remarks}

    It would be very nice to have a direct proof of the implication \autoref{thm:formally-etale-Noetherian-relative-Frobenius-iso.a} $\implies$ \autoref{thm:formally-etale-Noetherian-relative-Frobenius-iso.c} in \autoref{thm:formally-etale-Noetherian-relative-Frobenius-iso} using the fact that a locally nilpotent ideal in a Noetherian ring is nilpotent, and without using the deep results of Grothendieck, Radu, Andr\'e. However, we have been unsuccessful so far except in the case that $R, S$ are both $F$-finite: Suppose $R \to S$ is a formally smooth (resp. \'etale) map of $F$-finite $\mathbf{F}_p$-algebras such that $S$ is Noetherian, and suppose we are given a solid commutative diagram \autoref{equn-thediagram} in which $I^{[p]} = 0$. Let $B = \im(S \to A/I)$ and $B'$ be the inverse image of $B$ along $A \twoheadrightarrow A/I$. Then $B' \to B$ is surjective and $\ker(B' \twoheadrightarrow B) = B' \cap I$ also satisfies $\ker(B' \twoheadrightarrow B)^{[p]} = 0$. Note that any lift $S \to A$ of $S \to A/I$ along $A \twoheadrightarrow A/I$ must factor via $B'$. Thus, one sees by replacing $A/I$ with $B$ and $A$ with $B'$ that to show there is a (resp. is a unique) dashed arrow $S \to A$ we may assume $A/I$ is Noetherian and $F$-finite and $S \to A/I$ is surjective. Then one further reduces to the case where $A$ is Noetherian by the following lemma. 

\begin{lemma}
\label{lem:F-finite-alternate-lemma}
    Let $A$ be an $\mathbf{F}_p$-algebra for some prime number $p>0$ and $I \subset A$ an ideal such that $I^{[p]} = 0$. Assume that $A/I$ is Noetherian and $F$-finite. Then for any finite collection of ring maps $\varphi_i : R_i \to A, i = 1, \dots , n$ where each $R_i$ is an $\mathbf{F}_p$-algebra which is $F$-finite, there is a subring $A' \subset A$ such that
    \begin{enumerate}
        \item $A'$ is Noetherian.\label{lem:F-finite-alternate-lemma.a}
        \item The composition $A' \hookrightarrow A \twoheadrightarrow A/I$ is surjective.\label{lem:F-finite-alternate-lemma.b}
        \item For all $i$, the image of $R_i \to A$ is contained in $A'$.\label{lem:F-finite-alternate-lemma.c}
        \item If $J \coloneqq A' \cap I = \ker(A' \twoheadrightarrow A/I)$, then $J^{[p]} = 0$, and so, $J$ is nilpotent.\label{lem:F-finite-alternate-lemma.d}
    \end{enumerate}
\end{lemma}

\begin{proof}
    For each $i$, let $\{r_{ij}\}$ be a finite set of generators of $R_i$ as an algebra over $R_i^p$. Let $a_1, \dots , a_m$ be lifts to $A$ of generators of $A/I$ as an algebra over its subring of $p^{th}$-powers. Consider the subring
    $$
    A' = A^p[r_{ij}, a_k] \subset A.
    $$
    Then by choice of the $r_{ij}$ and $a_k$ we see that \autoref{lem:F-finite-alternate-lemma.b} and \autoref{lem:F-finite-alternate-lemma.c} hold. Finally, since $I^{[p]} = 0$, we see that $F_*A$ is an $A/I$-algebra, and $A'$ is the $A/I$-subalgebra generated by the finitely many elements $r_{ij}, a_k$, hence \autoref{lem:F-finite-alternate-lemma.a} holds. Finally, by the injectivity of $A' \hookrightarrow A$ and the fact that $I^{[p]} = 0$, we get $J^{[p]} = 0$. Then $J$ is nilpotent, i.e., $J^n = 0$ for some $n \gg 0$, because $J$ is finitely generated. This shows \autoref{lem:F-finite-alternate-lemma.d}.
\end{proof}

In other words, we have a direct demonstration of the following:

\begin{proposition}
    \label{prop:equivalences-b-nil-F-finite-case}
    Let $\varphi \colon R \to S$ be a map of $F$-finite $\mathbf{F}_p$-algebras. Suppose $S$ is Noetherian. Then:
    \begin{enumerate}
        \item $\varphi$ is b-nil formally smooth if and only if $\varphi$ is formally smooth.\label{prop:equivalences-b-nil-F-finite-case.a}
        \item $\varphi$ is b-nil formally unramified if and only if $\varphi$ is formally unramified.\label{prop:equivalences-b-nil-F-finite-case.b}
        \item $\varphi$ is b-nil formally \'etale if and only if $\varphi$ is formally \'etale.\label{prop:equivalences-b-nil-F-finite-case.c}
    \end{enumerate}
\end{proposition}

\begin{proof}
    Again, the non-trivial assertions in all three cases is that formally smooth (resp. \'etale) implies b-nil formally smooth (resp. \'etale). The discussion above clearly implies that formally smooth implies b-nil formally smooth in our setting.

    To see that formally unramified implies b-nil formally \'unramified, we have to establish the uniqueness of the lift $S \to A$ given a commutative diagram 
    \[\begin{tikzcd}[cramped]
	{A/I} & S \\
	A & R
	\arrow["f"', from=1-2, to=1-1]
	\arrow[two heads, from=2-1, to=1-1]
	\arrow[from=2-2, to=1-2]
	\arrow[from=2-2, to=2-1]
\end{tikzcd}\]
such that $I^{[p]} = 0$. Assume $f_1, f_2 \colon S \to A$ are $R$-algebra lifts of $f$ along $A \twoheadrightarrow A/I$. We may assume $S \to A/I$ is surjective and so $A/I$ is Noetherian and $F$-finite by the discussion in the paragraph preceding \autoref{lem:F-finite-alternate-lemma}. Then apply that Lemma to the three maps $R \to A$, $f_1 \colon S \to A$ and $f_2 \colon S \to A$ to construct a Noetherian $R$-subalgebra $A' \subset A$ such that all these maps have image contained in $A'$, and the restriction of $A \to A/I$ to $A'$ is surjective and still has 
 a nilpotent ideal $J$ of $A'$ such that $f_1, f_2$ both have images contained in $A'$, $A'/J$ is a subring of $A/I$ and $f$ has image $A'/J$, and 
the induced maps $f_1', f_2' \colon S \to A'$ obtained by restricting the codomain of $f_1, f_2$ lift $S \twoheadrightarrow A/I =  A'/J$. By formal \'unramifiedness, $f'_1 = f'_2$, and so, $f_1 = f_2$.

Finally, a formally \'etale map in this setting is b-nil formally \'etale by the previous two cases.
\end{proof}

\begin{remark}
\label{rem:animated-relative-Frobenius}
    We record a result for animated Noetherian $\mathbf{F}_p$-algebras that was explained to us by Bhatt and that may be of interest to readers. The result is due to Lurie. Namely, let $A \to B$ be a map of Noetherian animated $\mathbf{F}_p$-algebras. Then $A \to B$ is formally \'etale if and only if the relative Frobenius $F_{B/A}$ is an isomorphism. Indeed, the variant for the $\mathbb{E}_{\infty}$-cotangent complex for Noetherian $\mathbb{E}_\infty$-rings follows from \cite[Proposition~3.5.5, Proposition~3.5.6]{LurieEllipticII}, and \cite[Remark~25.3.3.7, Proposition~25.3.4.2]{LurieSpectralAlgebraicGeometry} then imply that the animated/algebraic cotangent complex vanishes if and only if the $\mathbb{E}_\infty$ one does.
\end{remark}

\subsection{Formal lifting properties for ideal adic completions}
\label{sec:formal-smooth-completion}

We end with a discussion of when a map from a Noetherian ring $R$ to its completion $\widehat{R}$ with respect to an ideal $I$ is (b-nil) formally smooth. Note that while $R \to \widehat{R}$ is always formally \'etale \emph{with respect to the $I\widehat{R}$-adic topology} (see proof of \cite[\href{https://stacks.math.columbia.edu/tag/07ED}{Tag 07ED}]{stacks-project}), it is frequently not formally smooth. For instance, when $R$ is local and $I$ is its maximal ideal, $R \to \widehat{R}$ being formally smooth implies $R$ is a G-ring by Grothendieck's Theorem that formally smooth maps of Noetherian rings are regular \cite[$0_{IV}$, Corollaire~19.6.5, Th\`eor\'eme~19.7.1]{EGAIV_I}. We have furthermore already seen an example of an excellent regular local ring in characteristic $p>0$ whose map to its completion is not formally smooth (\autoref{rem:infinite-field-char-p-not-formally-smooth}). When $\mathbf{Q} \subset R$, it is in fact extremely rare for $R \to \widehat{R}$ to be formally smooth. This is illustrated by the following example.

\begin{example}
\label{rem:very-rare-char0}
 Let $R = \mathbf{Q}[t]_{(t)}$. Then $\Omega_{\widehat{R}/R} \otimes _{\widehat{R}} \mathbf{Q} = 0$ since the module of K\"ahler differentials is compatible with base change and $R \to \widehat{R}$ induces an isomorphism of residue fields. However $\Omega_{\widehat{R}/R}[1/t] = \Omega_{\mathbf{Q}((t))/\mathbf{Q}(t)}$ which is an infinite-dimensional vector space over $\mathbf{Q}((t))$ since $\mathbf{Q}((t))/\mathbf{Q}(t)$ is an extension of fields of characteristic zero with infinite transcendence degree. Therefore, $\Omega_{\widehat{R}/R}$ is not $\widehat{R}$-free and therefore not projective, so $R \to \widehat{R}$ is not formally smooth by \cite[\href{https://stacks.math.columbia.edu/tag/031J}{Tag 031J}]{stacks-project} hence also not b-nil formally smooth.
\end{example}

Nevertheless, we have the following general result.

\begin{proposition}
\label{prop:equivalence-completions-arbitrary}
    Let $R$ be a ring and $I \subset R$ a finitely generated ideal. Let $\widehat{R}$ denote the $I$-adic completion of $R$. 
    \begin{enumerate}
        \item If $R \to \widehat{R}$ is formally smooth, then it is formally \'etale.\label{prop:equivalence-completions-arbitrary.a}

        \item Suppose $R$ is an $\mathbf{F}_p$-algebra. If $R \to \widehat{R}$ is b-nil formally smooth, then it is b-nil formally \'etale.\label{prop:equivalence-completions-arbitrary.b}

        \item \label{prop:equivalence-completions-arbitrary.c} Suppose $R$ is a Noetherian $\mathbf{F}_p$-algebra. Then the following are equivalent:
        \begin{enumerate}[(i)]
        \item $R \to \widehat{R}$ is formally smooth.
        \item $R \to \widehat{R}$ is formally \'etale.
        \item $R \to \widehat{R}$ is b-nil formally \'etale.
        \item $R \to \widehat{R}$ is b-nil formally smooth.
    \end{enumerate}
    \end{enumerate}
\end{proposition}

In order to prove \autoref{prop:equivalence-completions-arbitrary} we need the following preparatory lemmas.

\begin{lemma}
    \label{lem:relative-Frobenius-completion}
    Let $R$ be a ring of prime characteristic $p > 0$ and $I$ be a finitely generated ideal of $R$. Let $\widehat{R}$ denote the completion of $R$ with respect to $I$. Then the relative Frobenius $F_{\widehat{R}/R}$ can be identified with the canonical map associated with the completion of $\widehat{R}^{(p)} = F_*R \otimes_R \widehat{R}$ with respect to the finitely generated ideal $I\widehat{R}^{(p)} = (F_*I)\widehat{R}^{(p)}$.
\end{lemma}

\begin{proof}
    Note that the composition $R \to \widehat{R}^{(p)} \xrightarrow{F_{\widehat{R}/R}} \widehat{R}$ is the canonical $I$-adic completion map of $R$. Under the composition, if $J \coloneqq I\widehat{R}^{(p)}$, then $J\widehat{R} = I\widehat{R}$ and the map $R/I \to \widehat{R}^{(p)}/I\widehat{R}^{(p)}$ is an isomorphism. This is easiest to see using the $F_*$ notation. Indeed, the exact sequence $F_*I \to F_*R \to F_*(R/I) \to 0$ then gives us upon tensoring by $\otimes_R \widehat{R}$ that $\widehat{R}^{(p)}/I\widehat{R}^{(p)} \cong F_*(R/I) \otimes_R \widehat{R} \cong F_*(R/I) \otimes_{R/I} \widehat{R}/I\widehat{R} \cong F_*(R/I)$, since $\widehat{R}/I\widehat{R} \cong R/I$ by the fact that $I$ is finitely generated (\cite[\href{https://stacks.math.columbia.edu/tag/05GG}{Tag 05GG}]{stacks-project}). Thus, the lemma follows by the following more general fact: Let $\ba$ be a finitely generated ideal of a ring $A$ and $\widehat{A}^\ba$ denote the $\ba$-adic completion. If we have a factorization $A \to B \to \widehat{A}^\ba$ of the canonical completion map $A \to \widehat{A}^\ba$ and $A/\ba \to B/\ba B$ is an isomorphism, then $B \to \widehat{A}^\ba$ can be identified with $B \to \widehat{B}^{\ba B}$. Indeed, we obtain maps
    $$
    A/\ba ^n \to B/\ba^nB \to \widehat{A}^\ba/\ba^n \widehat{A}^\ba \cong  A/\ba ^n
    $$
    for all $n$ (the last isomorphism follows because $\ba$ is finitely generated) which compose to the identity and are compatible with the projections $A/\ba^{n+1} \twoheadrightarrow A/\ba^n$ and $B/\ba^{n+1}B \twoheadrightarrow B/\ba^nB$, and hence an injective map 
    $$\widehat{A}^\ba = \operatorname{lim}_n A/\ba ^n \to \operatorname{lim}_n B/\ba^n B = \widehat{B}^{\ba B}.$$
    It is also surjective by \cite[\href{https://stacks.math.columbia.edu/tag/0315}{Tag 0315}]{stacks-project} since $A/\ba \to B/\ba B$ is surjective, hence $\widehat{A}^\ba \to \widehat{B}^{\ba B}$ is an isomorphism such that the composition $B \to \widehat{A}^{\ba} \xrightarrow{\simeq} \widehat{B}^{\ba B}$ is the canonical completion map.
\end{proof}

\begin{lemma}
\label{lem:projective-Jacobson-radical}
    Let $R$ be a ring. Let $I \subset R$ be an ideal contained in the Jacobson radical of $R$. Let $P$ be a projective $R$-module. If $P/IP = 0$, then $P = 0$.
\end{lemma}

\begin{proof}
    Let $\mathfrak{m} \subset R$ be a maximal ideal. Then $P_\mathfrak{m}$ is a projective module over the local ring $R_{\mathfrak{m}}$ hence free by Kaplansky's Theorem. In fact, $P_{\mathfrak{m}} = 0$ since it is free and 
    $$
    P_{\mathfrak{m}} \otimes _R R/\mathfrak{m} = P_\m \otimes _R R/I \otimes _{R/I} R / \mathfrak{m} = 0
    $$
    since $I \subset \mathfrak{m}$ and $P/IP = 0$. Since $P_{\mathfrak{m}} = 0$ for every maximal ideal of $R$, we have $P = 0$ as needed. 
\end{proof}

\begin{proof}[Proof of \autoref{prop:equivalence-completions-arbitrary}]
    The equivalence of the assertions of \autoref{prop:equivalence-completions-arbitrary.c} for Noetherian $\mathbf{F}_p$-algebras follows from \autoref{prop:equivalence-completions-arbitrary.a} and \autoref{thm:formally-etale-Noetherian-relative-Frobenius-iso}. 

    We next show \autoref{prop:equivalence-completions-arbitrary.a}. Let $R$ is an arbitrary ring and $I \subset R$ is a finitely generated ideal. Then $I\widehat{R}$ is contained in the Jacobson radical of $\widehat{R}$ by \cite[\href{https://stacks.math.columbia.edu/tag/05GI}{Tag 05GI}]{stacks-project} (this does not need $I$ finitely generated). Since $I$ is finitely generated, we have that the induced map $R/I \to \widehat{R}/I\widehat{R}$ is an isomorphism by \cite[\href{https://stacks.math.columbia.edu/tag/05GG}{Tag 05GG}]{stacks-project}.
    Let $R \to \widehat{R}$ be formally smooth.
    It suffices to show $R \to \widehat{R}$ is formally unramified, that is, $\Omega_{\widehat{R}/R} = 0$. Since $R \to \widehat{R}$ is formally smooth, $\Omega_{\widehat{R}/R}$ is a projective $\widehat{R}$-module  \cite[\href{https://stacks.math.columbia.edu/tag/031J}{Tag 031J}]{stacks-project} which satisfies
    $$
    \Omega_{\widehat{R}/R} \otimes_{\widehat{R}} \widehat{R}/I\widehat{R} \cong \Omega_{\widehat{R}/R}/I\Omega_{\widehat{R}/R} \cong \Omega_{\widehat{R}/R} \otimes_{R }R/I =\Omega_{(\widehat{R}/I\widehat{R} ) / (R/I)} = 0,
    $$
    where the penultimate equality follows by the compatibility of K\"ahler differentials with base change. We then have $\Omega_{\widehat{R}/R} = 0$ by \autoref{lem:projective-Jacobson-radical}.

    We give two proofs of \autoref{prop:equivalence-completions-arbitrary.b}. For the first proof, consider the exact sequence
    $$
    0 \to \widehat{R}^{(p)} \xrightarrow{F_{\widehat{R}/R}} \widehat{R} \to \operatorname{Coker}(F_{\widehat{R}/R}) \to 0 
    $$
    of $\widehat{R}^{(p)}$-modules. Then $\widehat{R}$ is a projective $\widehat{R}^{(p)}$-module by Theorem \ref{thm:b-nil-formally-smooth-fp-relFrob} and $F_{\varphi}$ splits by \autoref{lem:projective-implies-split}. Thus $\operatorname{Coker}(F_{\widehat{R}/R})$ is also a projective $\widehat{R}^{(p)}$-module. Furthermore, since relative Frobenius is compatible with base change, we have 
    $$
    \operatorname{Coker}(F_{\widehat{R}/R}) \otimes _{R} R/I = \operatorname{Coker}(F_{(\widehat{R}/I\widehat{R})/(R/I)}) = 0
    $$
    since $R/I \to \widehat{R}/I\widehat{R}$ is an isomorphism. Now we conclude again by \autoref{lem:projective-Jacobson-radical} and the fact that $I \widehat{R}^{(p)}$ is contained in the Jacobson radical of $\widehat{R}^{(p)}$: There is a commutative square of ring maps
\[\begin{tikzcd}
	{\widehat{R}} & {\widehat{R}/I\widehat{R}}\\
	{\widehat{R}^{(p)}} & {\widehat{R}^{(p)}/I\widehat{R}^{(p)}} 
	\arrow[from=1-1, to=1-2]
	\arrow[from=2-1, to=1-1]
	\arrow[from=2-1, to=2-2]
	\arrow[from=2-2, to=1-2]
\end{tikzcd}\]
where the vertical arrows are integral and induce surjections (even homeomorphisms) on spectra. Since $I \widehat{R}$ is contained in the Jacobson radical of $\widehat{R}$ it follows that $I\widehat{R}^{(p)}$ is contained in the Jacobson radical of $\widehat{R}^{(p)}$ because for an integral map that is surjective on spectra, the Jacobson radical contracts to the Jacobson radical.

    Second proof of \autoref{prop:equivalence-completions-arbitrary.b}: Suppose $R \to \widehat{R}$ is b-nil formally smooth. Then $F_{\widehat{R}/R}$ splits as a map of $\widehat{R}^{(p)}$-modules by \autoref{prop:b-nil-formally-smooth-injective-relative-Frobenius-projective}. Since $F_{\widehat{R}/R}$ can also be identified as the canonical map associated with the completion of $\widehat{R}^{(p)}$ with respect to the finitely generated ideal $I\widehat{R}^{(p)}$ (\autoref{lem:relative-Frobenius-completion}), then $F_{\widehat{R}/R}$ is an isomorphism, and hence $R \to \widehat{R}$ is b-nil formally \'etale, by the following lemma.
\end{proof}

\begin{lemma}
    \label{lem:completion-never-split}
    Let $R$ be a ring and $I$ be a finitely generated ideal such that the canonical map $i \colon R \to \widehat{R}$ to the $I$-adic completion $\widehat{R}$ has an $R$-linear left-inverse. Then $i$ is an isomorphism.
\end{lemma}

\begin{proof}
    The argument essentially appears in \cite[Theorem~6.1.3]{DattaMurayamaFsolidity} but with the hypothesis that $R$ is Noetherian. So we include a proof for the reader's convenience. Let $\varphi \colon \widehat{R} \to R$ be an $R$-linear left inverse of $i$. It is enough to show $i \circ \varphi = \id_{\widehat{R}}$. We have $(i \circ \varphi - \id_{\widehat{R}}) \circ i = 0$. 
    Now consider any $f \in \Hom_R(\widehat{R},R)$ such that $f \circ i = 0$. The result follows if we can show that $f = 0$.
    Since $I$ is finitely generated, for all $n \geq 0$, the map $i \otimes_R \id_{R/I^n} \colon R/I^n \to \widehat{R}/I^n\widehat{R}$ is an isomorphism by \cite[\href{https://stacks.math.columbia.edu/tag/05GG}{Tag 05GG}]{stacks-project}. Thus, for all $n \geq 0$, $\widehat{R} = \im(i) + I^n\widehat{R}$. Let $x \in \widehat{R}$ and for all $n \geq 0$ choose $a_n \in R$ and $x_n \in I^n\widehat{R}$ such that $x = i(a_n) + x_n$. Then $f(x) = f(i(a_n) + x_n) = f(x_n) \in I^n$ by the linearity of $f$. Thus, $f(x) \in \bigcap_{n \geq 0} I^n = \ker(i) = 0$, and so, $f = 0$.
\end{proof}

\begin{remark}
    Let $R$ be a non-Noetherian ring. If we relax the assumption for the ideal $I$ to be finitely generated, then $R \to \widehat{R}$ can be formally \'etale without being b-nil formally \'etale. Indeed, consider the example of the $\mathbf{F}_p$-algebra $R$ of \autoref{eg:formally-etale-not-b-nil formally-etale}~\autoref{eg:formally-etale-not-b-nil formally-etale.b} with an ideal $I$ such that $I = I^2$ and $I \neq I^{[p]}$. Then $\widehat{R} = R/I$ and one can identify $R \to \widehat{R}$ with the canonical projection $R \twoheadrightarrow R/I$, which as we have seen is formally \'etale but not b-nil formally \'etale.  Note that such an ideal $I$ will necessarily be non-finitely generated as otherwise the condition $I = I^2$ would imply that $R$ is principally generated by an idempotent of $R$ and consequently that $I = I^{[p]}$. Thus, it is conceivable that \autoref{prop:equivalence-completions-arbitrary}~\autoref{prop:equivalence-completions-arbitrary.c} holds even without the hypothesis that $R$ is Noetherian. We do not know if this is true.
\end{remark}

We have seen several examples of the form where the completion map $R \to \widehat{R}$ is regular but not formally smooth. One of them was the completion map $k[t]_{(t)} \to k[[t]]$ for a non-$F$-finite field $k$ of characteristic $p$ (\autoref{rem:infinite-field-char-p-not-formally-smooth}). We next show that the non-$F$-finite field assumption is essential. Related observations were made by the first author and Remy van Dobben de Bruyn and also by  the first author and Linquan Ma for completions of Noetherian local rings with respect to their maximal ideals (both works are unpublished). See also \autoref{rem:following-prop-completion-b-nil-formally-etale} for related results obtained by Seydi \cite{SeydiExcellentII}.

\begin{proposition}
\label{prop:completionforF-finiteGring}
    Let $R$ be a Noetherian ring of prime characteristic $p > 0$. Let $I$ be an ideal of $R$ such that $R/I$ is $F$-finite. Let $\widehat{R}$ denote the $I$-adic completion of $R$. Then we have the following:
    \begin{enumerate}
    \item $\widehat{R}$ is $F$-finite.\label{prop:completionforF-finiteGring.a0}
    \item $F_{\widehat{R}/R}$ is surjective, and hence, $R \to \widehat{R}$ is b-nil formally unramified.\label{prop:completionforF-finiteGring.a}
    \item \label{prop:completionforF-finiteGring.b} The following are equivalent:
        \begin{enumerate}[(i)]
        \item $R \to \widehat{R}$ is a regular map.\label{prop:completionforF-finiteGring.ii}
        \item $R \to \widehat{R}$ is b-nil formally \'etale.\label{prop:completionforF-finiteGring.iii}
        \item $R \to \widehat{R}$ is formally \'etale.\label{prop:completionforF-finiteGring.iv}
        \item $R \to \widehat{R}$ is formally smooth.\label{prop:completionforF-finiteGring.v}
        \item $R \to \widehat{R}$ is b-nil formally smooth.\label{prop:completionforF-finiteGring.vi}
        \item $F_{\widehat{R}/R}$ is injective.\label{prop:completionforF-finiteGring.vii}
        \end{enumerate}
    \end{enumerate}
    If $I$ is contained in the Jacobson radical of $R$, then \autoref{prop:completionforF-finiteGring.ii}-\autoref{prop:completionforF-finiteGring.vii} are equivalent to $R$ being $F$-finite.
\end{proposition}

Thus, \autoref{prop:completionforF-finiteGring} implies that if $k$ is an $F$-finite field, then $k[t]_{(t)} \to k\llbracket t\rrbracket$ is b-nil formally \'etale, unlike \autoref{eg:regular-not-bil-nil-smooth}.

\begin{proof}[Proof of \autoref{prop:completionforF-finiteGring}]
\autoref{prop:completionforF-finiteGring.a0} is well-known but for lack of a precise reference we include a proof here. Since the ideals $I^{[p^e]}$ for $e \geq 0$ are co-final with the ordinary powers of $I$, the $I$-adic completion, $\widehat{F_*R}$, of $F_*R$ coincides with $F_*\widehat{R}$. Thus, in order to show $F_*\widehat{R}$ is a finitely generated $\widehat{R}$-module, it suffices to show by \cite[\href{https://stacks.math.columbia.edu/tag/031D}{Tag 031D}]{stacks-project} that $F_*\widehat{R}/IF_*\widehat{R} \cong \widehat{F_*R}/I\widehat{F_*R} \cong F_*R/IF_*R \cong F_*(R/I^{[p]})$ is a finitely generated $\widehat{R}/I\widehat{R} \cong R/I$-module. Here the isomorphisms follow by \cite[\href{https://stacks.math.columbia.edu/tag/05GG}{Tag 05GG}]{stacks-project} because $I$ is a finitely generated ideal. Recall that a prime characteristic Noetherian ring $A$ is $F$-finite if and only if $A_{\red}$ is $F$-finite (for example, see \cite[Exercise~4]{MaPolstraFSing}). Thus, $F$-finiteness of $R/I$ implies $(R/I)_{\red} = (R/I^{[p]})_{\red}$ is $F$-finite, and so, $R/I^{[p]}$ is $F$-finite. The absolute Frobenius of $R/I^{[p]}$ is the composition $R/I^{[p]} \twoheadrightarrow R/I \to F_*(R/I^{[p]})$, where the first map is the canonical projection and the second one sends $R/I \ni r+ I \mapsto r^p + I^{[p]} \in F_*(R/I^{[p]})$. Thus, $F_*(R/I^{[p]})$ is a finitely generated $R/I$-module. 

\autoref{prop:completionforF-finiteGring.a} We have shown in the proof of \autoref{lem:relative-Frobenius-completion} that one can identify $F_{\widehat{R}/R}$ with the canonical map associated with the completion of $\widehat{R}^{(p)}$ with respect to the finitely generated ideal $I\widehat{R}^{(p)}$. By \autoref{prop:completionforF-finiteGring.a0} we see $\widehat{R}$ is $F$-finite. Thus, $F_{\widehat{R}/R}$ is a finite map. Then $F_{\widehat{R}/R}$ is surjective by \cite[Remark~6.1.5~(2)]{DattaMurayamaFsolidity} and the proof of \cite[Lemma~6.1.4~(iii)]{DattaMurayamaFsolidity} which show that if $A$ is \emph{any} ring and $\ba$ is a finitely generated ideal of $A$, then the canonical map $A \to \widehat{A}^{\ba}$ is finite if and only if it is surjective.


We now prove the equivalent assertions of \autoref{prop:completionforF-finiteGring.b}. The equivalence of \autoref{prop:completionforF-finiteGring.iii}, \autoref{prop:completionforF-finiteGring.iv}, \autoref{prop:completionforF-finiteGring.v} and \autoref{prop:completionforF-finiteGring.vi} follows by \autoref{prop:equivalence-completions-arbitrary} because $R$ is Noetherian. Clearly, \autoref{prop:completionforF-finiteGring.iii}$\implies$\autoref{prop:completionforF-finiteGring.vii} by \autoref{thm:relative-Frobenius-iso-b-nil formally-etale}. Also, \autoref{prop:completionforF-finiteGring.ii}$\implies$\autoref{prop:completionforF-finiteGring.iii} because $F_{\widehat{R}/R}$ is faithfully flat and hence injective by \autoref{thm:Radu-Andre}~\autoref{thm:Radu-Andre.a}, hence an isomorphism by \autoref{prop:completionforF-finiteGring.a}. It remains to show \autoref{prop:completionforF-finiteGring.vii}$\implies$\autoref{prop:completionforF-finiteGring.ii}. Again, $F_{\widehat{R}/R}$ is an isomorphism by \autoref{prop:completionforF-finiteGring.a} and hence $R \to \widehat{R}^I$ is regular by \autoref{thm:Radu-Andre}~\autoref{thm:Radu-Andre.a}.

Finally, suppose $I$ is contained in the Jacobson radical of $R$. Then $R \to \widehat{R}^I$ is faithfully flat. Now, $\widehat{R}$ is $F$-finite by \autoref{prop:completionforF-finiteGring.a0}. Let us assume \autoref{prop:completionforF-finiteGring.ii}. Then $\widehat{R}^{(p)}$ is a $\widehat{R}$-submodule (by faithful flatness of $F_{\widehat{R}/R}$) of the finitely generated $\widehat{R}$-module $F_*\widehat{R}$. Thus, $R \to \widehat{R}^{(p)}$ is finite since $\widehat{R}$ is Noetherian. But this finite map is the base change of the absolute Frobenius $F_R$ along the faithfully flat map $R \to \widehat{R}$, so $F_R$ is also finite by descent \cite[\href{https://stacks.math.columbia.edu/tag/08XD}{Tag 08XD}]{stacks-project}.
\end{proof}

\begin{remarks}
\label{rem:following-prop-completion-b-nil-formally-etale}
{\*}
\begin{enumerate}
    \item If $I$ is a finitely generated ideal of a non-Noetherian ring $R$, then as long as we assume that $F_*(R/I^{[p]})$ is a finitely generated $R/I$-module (this is stronger than assuming $R/I$ is $F$-finite) the conclusions of \autoref{prop:completionforF-finiteGring}~\autoref{prop:completionforF-finiteGring.a0} and \autoref{prop:completionforF-finiteGring}~\autoref{prop:completionforF-finiteGring.a} will continue to hold.
    
    \item Non-excellent/non-$F$-finite regular local rings with $F$-finite residue fields exist even in the function field of $\mathbf{P}^2_{\mathbf{F}_p}$. See for instance \cite{DattaSmithExcellence}. For any such ring $(R,\m)$, if $\widehat{R}$ is the $\m$-adic completion, then $F_{\widehat{R}/R}$ is surjective by \autoref{prop:completionforF-finiteGring}~\autoref{prop:completionforF-finiteGring.a} but not injective by \autoref{prop:completionforF-finiteGring}~\autoref{prop:completionforF-finiteGring.b}. \label{rem:non-excellent-F-finite-residuefield}

    \item Suppose $R$ is Noetherian local and $I$ is the maximal ideal of $R$. One can use \cite[Proposition~3.3]{SeydiExcellentII} to also deduce parts of \autoref{prop:completionforF-finiteGring}~\autoref{prop:completionforF-finiteGring.b}. However, Seydi's argument in the local case is quite different than our proof of \autoref{prop:completionforF-finiteGring}.
    

    \item Let $(R,\m,\kappa)$ be a Noetherian local ring. Although $R \to \widehat{R}$ is formally \'etale in many geometric situations in prime characteristic (\autoref{prop:completionforF-finiteGring}), this map is rarely weakly \'etale. Indeed, let $(R^h,\m^h)$ denote the Henselization of $(R,\m)$.  Note, $\m^h = \m R^h$. We have canonical faithfully flat local maps $R \to R^h$ and $R^h \to \widehat{R}$ such that the composition $R \to R^h \to \widehat{R}$ is the canonical map associated with completion (this follows by the universal property of Henselization and the fact that $\widehat{R}$ is Henselian). Since $(R^h,\m^h)$ is a Henselian pair, if $R \to \widehat{R}$ is weakly \'etale, then by \cite[Theorem~1]{deJongOlander}, there exists a unique map $\widehat{R} \to R^h$ such that the following diagram commutes:
    \[\begin{tikzcd}[cramped]
	\kappa & {\widehat{R}} \\
	{R^h} & R
	\arrow[two heads, from=1-2, to=1-1]
	\arrow["{\exists!}"{description}, color={rgb,255:red,214;green,92;blue,92}, dashed, from=1-2, to=2-1]
	\arrow[two heads, from=2-1, to=1-1]
	\arrow[from=2-2, to=1-2]
	\arrow[from=2-2, to=2-1].
\end{tikzcd}\]
Using the universal property of Henselizations and completions, it follows that $\widehat{R} \to R^h$ is the inverse of $R^h \to \widehat{R}$, that is, $R^h \to \widehat{R}$ is an isomorphism. But $R^h$ is rarely isomorphic to $\widehat{R}$ for a non-complete Noetherian local ring $R$. For example, take $R = \mathbf{F}_p[t]_{(t)}$. Then $\widehat{R} = \mathbf{F}_p\llbracket t \rrbracket$ is uncountable. However, $R^h$ is countable (its fraction field is an algebraic extension of $F_p(t)$, which is countable), so it cannot be isomorphic to $\widehat{R}$. Note in this case $\mathbf{F}_p[t]_{(t)} \to \mathbf{F}_p\llbracket t \rrbracket$ is b-nil formally \'etale by \autoref{prop:completionforF-finiteGring}.\label{rem:completion-weakly-etale}

\item There is often a misconception that ind-smooth implies formally smooth, probably because ind-formally \'etale implies formally \'etale. This fact is false as shown by \autoref{rem:very-rare-char0} and \autoref{rem:infinite-field-char-p-not-formally-smooth}. These remarks show that ind-smooth maps fail to formally smooth even for the completion map of excellent DVRs. Note that the completion map of an excellent local ring is regular and hence ind-smooth by N\'eron-Popescu desingularization.
\end{enumerate}
\end{remarks}

\section{Acknowledgements}
\autoref{thm:formally-etale-F-finite} was inspired by a conversation with Bhargav Bhatt. We also thank him for comments on a draft, for discussions related to \autoref{eg:formally-etale-not-b-nil formally-etale} and for explaining to us \autoref{rem:animated-relative-Frobenius}. The first author also thanks Benjamin Antieau and Linquan Ma for conversations over the years on the absolute/relative Frobenius. He especially thanks Linquan Ma for discussions related to \autoref{prop:completionforF-finiteGring}. The second author thanks Benjamin Antieau, Matthew Morrow, and other participants in the 2019 Arizona Winter School for conversations about variants of formal smoothness. Additionally, we thank Benjamin Antieau and Aise Johan de Jong for comments and their interest. The first author was supported in part by NSF DMS Grant \#2502333 and Simons Foundation grant MP-TSM-00002400. The second author was supported in part by NSF DMS Grant \#2402087. 

\section{Future work}
The notion of b-nil formally unramified morphism has not really been explored in this paper. After the first version of the paper was released on arXiv, we learned that Javier Carvajal-Rojas and Axel St{\"a}bler were simultaneously working on (b-nil) formal unramifiedness. Thus, we will be releasing a joint paper with them where we give a complete characterization of this notion for maps of $\mathbf{F}_p$-algebras. To give a glimpse of the results we have, we have seen that the relative Frobenius being an epimorphism implies b-nil formal unramifiedness (\autoref{rem:surjective-rel-Frobenius-b-nil-formally-unramified}). It turns out the converse is also true in general, i.e., b-nil formal unramifiedness characterizes when the relative Frobenius is an epimorphism. This and other related results will appear in the future work with Carvajal-Rojas and St{\"a}bler.

\bibliographystyle{skalpha}
\bibliography{main,preprints}

\end{document}